\documentclass[a4,twoside, 10pt]{article}

\usepackage[a4paper, textwidth=15.75cm, textheight=23.4cm, marginratio={5:5,5:6},marginparwidth=2.1cm]{geometry}

\newcommand{\myauthor}{Jeremiah Heller, Mircea Voineagu, Paul Arne {\O}stv{\ae}r}

\newcommand{\mytitle}{Bredon motivic cohomology of the complex numbers}

\title{\mytitle}
\author{  
Jeremiah Heller \footnote{Heller was supported by NSF award DMS-1710966.},
Mircea Voineagu,
Paul Arne {\O}stv{\ae}r 
\footnote{{\O}stv{\ae}r gratefully acknowledges the support of the Centre for Advanced Study at the Norwegian Academy of Science and Letters in Oslo,
Norway, which funded and hosted the research project ``Motivic Geometry" during the 2020/21 academic year,
RCN Frontier Research Group Project no. 250399 ``Motivic Hopf Equations",  no. 312472 ``Equations in Motivic Homotopy",  
and Alexander von Humboldt Foundation.
}} 
\date{}

\usepackage{amsmath}
\usepackage{amscd}
\usepackage{amsthm} 
\usepackage{amssymb}
\usepackage{amsfonts}
\usepackage{eucal}
\usepackage{microtype}
\usepackage{mathrsfs}
\usepackage{dsfont}
\usepackage{stmaryrd}
\usepackage[margin = 1.5cm]{caption}
\usepackage[shortlabels]{enumitem}
\setlist[enumerate]{font=\upshape, topsep=0.5pt,itemsep=-0.5ex,partopsep=1ex,parsep=1ex} 
\setlist[enumerate,1]{label = (\arabic*)}
\setlist[1]{labelindent=\parindent} 
\setlist[itemize,1]{topsep=0.5pt,itemsep=-0.5ex,partopsep=1ex,parsep=1ex} 
\usepackage{fancyhdr}
\usepackage[dvipsnames]{xcolor} 
\definecolor{refkey}{gray}{0.5}
\definecolor{labelkey}{gray}{0.5}
\definecolor{note}{rgb}{0.94, 0.99, 1.00}
\colorlet{myurlcolor}{Aquamarine}
\colorlet{mylinkcolor}{violet}
\colorlet{mycitecolor}{YellowOrange}
\definecolor{reference}{rgb}{.93,.51,.93}
\definecolor{citation}{rgb}{1,.68,.26}
\definecolor{mrnumber}{rgb}{.80,.40,0}
\newcommand\myshade{85}

\usepackage{graphicx}
\usepackage{float}
\usepackage[all]{xy} %To use xy-pic for diagrams.
\usepackage{tikz}
\usetikzlibrary{patterns}
\tikzstyle{mygrid}=[gray!25!white]

\tikzstyle{myfill} = [fill=gray, fill opacity= 0.2]
\tikzstyle{myhatch} = [pattern=north west lines, pattern color=blue!35]
\tikzstyle{myfill1} = [fill=green!70!white, fill opacity= 0.1]
\tikzstyle{myfill2} = [fill=blue!70!white, fill opacity= 0.1]
\tikzstyle{myfill3} = [fill=red!70!white, fill opacity= 0.1]
\tikzstyle{myline} = [gray!40]
\usetikzlibrary{cd}
\usetikzlibrary{plotmarks}
\usetikzlibrary{arrows}
\makeatletter                                    %%
\g@addto@macro\@floatboxreset\centering          %%
\makeatother                                     %%
\usepackage{subcaption}

\usepackage[pdfstartview=FitH,
pdfauthor={\myauthor},
pdftitle={\mytitle},
colorlinks,
linkcolor=mylinkcolor!\myshade!black,
citecolor=citation!\myshade!black,
urlcolor=myurlcolor!\myshade!black,
backref=page,
bookmarks, 
bookmarksopen=true, 
bookmarksopenlevel=3, 
bookmarksdepth=3]{hyperref}

\usepackage[textsize=scriptsize, color=note, linecolor=lightgray]{todonotes}
\newcommand{\aref}[1]{\autoref{#1}}

\usepackage{aliascnt}

\newcommand{\C}{\mathds{C}} 
\newcommand{\A}{\mathds{A}}
\renewcommand{\P}{\mathds{P}}
\newcommand{\Z}{\mathds{Z}}

\newcommand{\E}{\mathrm{E}}
\newcommand{\XX}{\mathcal{X}}

\newcommand{\rH}{\widetilde{H}}
\newcommand{\Br}{\mathrm{Br}}

\newcommand{\pt}{\mathrm{pt}}

\renewcommand{\Re}{\mathrm{Re}}

\newcommand{\Sm}{\mathrm{Sm}}

\newcommand{\SH}{\mathrm{SH}}

\newcommand{\NC}{\mathrm{NC}}
\newcommand{\block}{\mathrm{B}}

\newcommand{\EG}{\mathbf{E}}
\newcommand{\BG}{\mathbf{B}}
\newcommand{\EGt}{\widetilde{\EG} }
\newcommand{\Es}{\mathrm{E}}
\newcommand{\Bs}{\mathrm{B}}

\newcommand{\wt}[1]{\widetilde{#1}}

\newcommand{\mcal}[1]{\mathcal{#1}}
\newcommand{\ul}[1]{\underline{\smash{#1}}}

\newcommand{\sing}{\mathrm{sing}}

\newcommand{\xto}[1]{\xrightarrow{#1}}

\newcommand{\MMt}{\mathds{M}^{C_2}_2} 
\newcommand{\MMtn}{\mathds{M}^{C_2}_n}

\DeclareMathOperator*{\colim}{\mathrm{colim}}

\DeclareMathOperator{\spec}{\mathrm{Spec}}

\newcommand{\cd}{\smash\cdot}

\makeatletter
\newcommand{\subalign}[1]{%
	\vcenter{%
		\Let@ \restore@math@cr \default@tag
		\baselineskip\fontdimen10 \scriptfont\tw@
		\advance\baselineskip\fontdimen12 \scriptfont\tw@
		\lineskip\thr@@\fontdimen8 \scriptfont\thr@@
		\lineskiplimit\lineskip
		\ialign{\hfil$\m@th\scriptstyle##$&$\m@th\scriptstyle{}##$\hfil\crcr
			#1\crcr
		}%
	}%
}
\makeatother
\numberwithin{equation}{section} %Fiddles with numbering system of the following.

\theoremstyle{plain} 
 
\newaliascnt{theorem}{equation}  
\newtheorem{theorem}[theorem]{Theorem}  
\aliascntresetthe{theorem}

\newaliascnt{proposition}{equation}  
\newtheorem{proposition}[proposition]{Proposition}
\aliascntresetthe{proposition}

\newaliascnt{lemma}{equation}    
\newtheorem{lemma}[lemma]{Lemma}
\aliascntresetthe{lemma}

\newaliascnt{corollary}{equation}  
\newtheorem{corollary}[corollary]{Corollary}
\aliascntresetthe{corollary}

\newaliascnt{claim}{equation}  

\aliascntresetthe{claim}

 \theoremstyle{definition}

\newaliascnt{definition}{equation}  
\newtheorem{definition}[definition]{Definition}
\aliascntresetthe{definition}

\newaliascnt{example}{equation}  
\newtheorem{example}[example]{Example}
\aliascntresetthe{example}

\newaliascnt{remark}{equation}   
\newtheorem{remark}[remark]{Remark}
\aliascntresetthe{remark}

\newaliascnt{condition}{equation}

\aliascntresetthe{condition}

\newaliascnt{notationconvention}{equation}

\aliascntresetthe{notationconvention}

\newaliascnt{notation}{equation}
\newtheorem{notation}[notation]{Notation}
\aliascntresetthe{notation}

\begin{document}

\maketitle
\begin{abstract}
Over the complex numbers, 
we compute the $C_2$-equivariant Bredon motivic cohomology ring with $\Z/2$ coefficients.
By rigidity, 
this extends Suslin's calculation of the motivic cohomology ring of algebraically closed fields of characteristic zero to the $C_2$-equivariant motivic setting. 


\paragraph{Mathematics Subject Classification 2010.}
Primary:
    \href{https://mathscinet.ams.org/msc/msc2010.html?t=14Fxx&btn=Current}{14F42},
    \href{https://mathscinet.ams.org/msc/msc2010.html?t=55Pxx&btn=Current}{55P91}.
Secondary:
    \href{https://mathscinet.ams.org/msc/msc2010.html?t=55Pxx&btn=Current}{55P42},
    \href{https://mathscinet.ams.org/msc/msc2010.html?t=55Pxx&btn=Current}{55P92}.

\end{abstract} 
 
\tableofcontents

\section{Introduction}
Bredon motivic cohomology (introduced in \cite{HOV} and \cite{HOV1}) is a generalization of motivic cohomology to the setting of smooth varieties with finite group action.   
Part of a larger group of motivic 
$C _2$-invariants, such as Hermitian K-theory and motivic real cobordism, it plays an essential role in equivariant motivic homotopy theory. 
One distinguishing feature is that Bredon motivic cohomology appears as the zero slice of the equivariant motivic sphere \cite{zeroslice}.

Bredon motivic cohomology is ready for concrete computations, which will be crucial for applications of the theory to other motivic and topological invariants. In this paper, we compute the Bredon motivic cohomology ring with $\Z/2$-coefficients. The usual methods \cite{Sus}, \cite{YO}, \cite{YP} generalize the computations to an algebraically closed field of characteristic zero. These can be seen as a first step in understanding the largely unknown and difficult to compute 
 Bredon cohomology ring for an arbitrary field $k$ (for partial results in this direction see \cite{Voi}) as well as the 
 $C_2$-equivariant motivic Steenrod algebra of cohomology operations.

Our computations are organized via modules over Bredon cohomology of a point. 
Before presenting our computations, we recall this ring and  introduce some notation used to explain our results. 

 \subsection{Bredon cohomology}\label{sec:tbc}

In equivariant topology, Bredon cohomology plays the role that singular cohomology plays in ordinary topology. Some of its key features are that it takes a Mackey functor as coefficients, it is graded by representations, and is represented by an equivariant Eilenberg-MacLane spectrum, see \cite{May:equi} for details. 
The case of interest to us is $G=C_2$, in which case we write $\sigma$ for the sign representation. The group ${\rm RO}(C_2)$ is identified with $\Z\oplus\Z\{\sigma\}$. We adopt the convention that $\star$ stands for an 
${\rm RO}(C_2)$-grading and we use $\ast$ for an integer grading. For an abelian group $A$, the Bredon cohomology with coefficients in the constant Mackey functor $\ul{A}$, of a $C_2$ spectrum $\mathcal{X}$ is written $\rH^{i+p\sigma}_{\Br}(\mathcal{X}, \ul{A})$. If $\mathcal{X} = \Sigma^{\infty}X_+$ we simply write 
 $H^{i+p\sigma}_{\Br}(X, \ul{A}) := \rH^{i+p\sigma}_{\Br}(\Sigma^{\infty}X_+, \ul{A})$.

The Bredon cohomology ring of a point with $\ul{\mathbb{Z}/2}$-coefficients was originally computed by Stong in unpublished work. Written accounts can be found in  \cite[Appendix]{Car:op} and \cite[Proposition 6.2]{HKr}.  For the corresponding computation with $\ul{\Z}$-coefficients see \cite[Theorem 2.8]{D} or \cite[Section 2]{Greenlees:4} for recent discussion of these computations. 
We write 
\[
\MMt:= H^{\star}_{\Br}(\pt,\ul{\Z/2}).
\]

Let $\Z/2[a,u]$ be the polynomial ring generated by elements whose degrees are $|a|=\sigma$ and $|u| = -1 +\sigma$. Consider 
$\Z/2[a^{-1}, u^{-1}]$ as a graded $\Z/2[a,u]$-module and write 
\[
\NC := \Sigma^{2-2\sigma}\Z/2[a^{-1}, u^{-1}], 
\]
where $\Sigma^{m+n\sigma}M$ denotes the shifted graded module given by  $(\Sigma^{m+n\sigma}M)^{a+p\sigma}=M^{a-m+(p-n)\sigma}$. From \aref{point2}, one sees that $\MMt$ consists of two cones; $\NC$ is the ``negative cone". The Bredon cohomology ring of a point is
\begin{equation}\label{eqn:Brofpt}
\MMt \cong \Z/2[a, u] \oplus \NC,
\end{equation}
where the multiplicative structure is determined by the action of $\Z/2[a, u]$ on $\NC$ and all products between elements in $\NC$ are trivial. 
Writing $\theta\in \NC$ for the element which corresponds to $1\in \Z/2[a^{-1}, u^{-1}]$, we  express elements of $\NC$ in the form $\frac{\theta}{a^mu^n}$, for $m,n\geq 0$.

	We  introduce some auxiliary $\MMt$-modules. See \aref{point2} for graphical depictions. Recall the universal free $C_2$-space $\Es C_2$; a geometric model is  $S({\infty}\sigma) = \colim_nS(n\sigma)$, where $S(n\sigma)$ is the unit sphere in the $n$-dimensional real sign representation.  The space $\wt\Es C_2$ is defined to be  the unreduced suspension of $\Es C_2$, which by definition fits into the cofiber sequence of based $C_2$-spaces, $\Es C_{2+} \to S^{0} \to \wt \Es C_2$.

The Bredon cohomology ring of $\Es C_2$ is
		\[
		H^\star_{\Br}(\Es C_2,\ul{\Z/2}) \cong \Z/2[a, u^{\pm 1}]
		\]
		where $|u| =  -1+\sigma$ and $|a| = \sigma$, see e.g., \cite[Lemma 27]{Car:op}. Then $	H^\star_{\Br}(\Es C_2,\ul{\Z/2})\cong \MMt[u^{-1}]$ and the ring map $\MMt\to H^{\star}_{\Br}(\Es C_2,\ul{\Z/2})$ is the localization map. In other words, it is the map
		which sends $u\mapsto u$, $a\mapsto a$ and maps  $\NC$ to $0$.

The Bredon cohomology of $\wt{\Es}C_2$ is
		\[
		\rH^\star_{\Br}(\wt{\Es }C_2,\ul{\Z/2}) \cong \Sigma^{2-2\sigma} \Z/2[ a^{\pm 1}, u^{-1}],
		\]
		 see e.g., \cite[Lemma 28]{Car:op}. The right-hand side is a $\Z/2[a,u]$-module and hence an $\MMt$-module  (elements of $\NC$ act by $0$) and this isomorphism is an $\MMt$-module isomorphism. 
The negative cone is a quotient of this module, 
\[
\NC \cong  \Sigma^{2-2\sigma} \Z/2[ a^{\pm 1}, u^{-1}]/a\cd\Z/2[a,u^{-1}].
\]	
		The $\MMt$-module map
		\begin{equation}\label{eqn:wtEtopt}
		 \Sigma^{2-2\sigma} \Z/2[ a^{\pm 1}, u^{-1}]\to \MMt,
		\end{equation}
		induced by $S^0\to \wt\Es C_2$, is this quotient followed by the inclusion of the negative cone into $\MMt$. Explicitly, it is the map
		\[
		\frac{\theta}{u^n}a^m\mapsto \begin{cases} 
			\displaystyle{\frac{\theta}{a^{-m}u^n} } & m\leq 0 \vspace{1mm}\\
			0 & m> 0
		\end{cases}
		\]
		where $\theta\in \Sigma^{2-2\sigma} \Z/2[a^{\pm 1}, u^{-1}]$ denotes the element corresponding to $1\in \Z/2[u^{-1}, x^{\pm 1}]$.

\begin{figure}[H]
\begin{subfigure}[b]{0.45\textwidth}
	\begin{tikzpicture}[scale=0.6]
\path[step=1.0,black,thin,xshift=0.5cm,yshift=0.5cm] (-6.5,-6) grid (6,6);
		\path[->, very thick] (6.3,0) edge (7.5,0);
		\draw[style=mygrid] (-5.0,-5.0) grid (5.0,5.0);
		\path[style=myfill] (0,0) -- (-5,5) -- (0,5) -- (0,0);
		\path[style=myfill] (2,-2) -- (5,-5) -- (2,-5) -- (2,-2);
		\draw[->, thin, gray!80] (-5.0,0) -- (5.3,0) ;
		\draw (5.5,0) node[black] {$i$};
		\draw[->, thin, gray!80] (0,-5.0) -- (0,5.3);
		\draw (0,5.5) node[black] {$p\cd \sigma$};
		\foreach \y in {0,...,5}{
		\foreach \x in {-\y,...,0}{
		\fill (\x,\y) circle (2.5pt);
	}
}
		\foreach \y in {-5,...,-2}{
		\foreach \x in {2,...,-\y}{
		\fill (\x,\y) circle (2.5pt);
	}
}		
		\draw (2,-4) node[left]{\small $\frac{\theta}{a^2}$};
		\draw (2,-5) node[left]{\small $\frac{\theta}{a^3}$};
		\draw (3.05,-3) node[right, ] {\small$\frac{\theta}{u}$};
		\draw (4.05,-4) node[ right]{\small $\frac{\theta}{u^2}$};
		\draw (5.05,-5) node[right]{ \small$\frac{\theta}{u^3}$};
		\draw (-1.05,1) node[left] {\small$u$};
		\draw(-2.05,2) node[left]{\small $u^2$};
		\draw(-3.05,3) node[left]{\small $u^3$};
		\draw(-4.05,4) node[left]{\small $u^4$};
		\draw (-5.05,5) node[left] {\small$u^5$};
		
		\draw (0,2) node[right]{\small $a^2$};
		\draw (0,3) node[right]{\small$a^3$};
		\draw (0,4) node[right]{\small$a^4$};
		\draw (0,5) node[right]{ \small$a^5$};
		
		\draw (0,1) node[right] {\small$a$};
		\draw (2,-2) node[above] {\small$\theta$};
		\draw (2,-3) node[left] {\small$\frac{\theta}{a}$};

		\draw[ red] (0, 0) -- (-5,5);
		\draw[ red] (0, 1) -- (-4,5);
		\draw[ red] (0, 2) -- (-3,5);
		\draw[ red] (0, 3) -- (-2,5);
		\draw[red] (0,4) -- (-1,5);
		\draw[ red] (2, -2) -- (5,-5);
		\draw[ red] (2, -3) -- (4,-5);
		\draw[red] (2, -4)	--(3,-5);
		
		\draw[blue] (0,0)  to  (0,5);
		\draw[blue] (-1,1)  to  (-1,5);
		\draw[blue] (-2,2)  to  (-2,5); 
		\draw[blue] (-3,3)  to  (-3,5);
		\draw[blue] (-4,4) to (-4,5);
		\draw[blue] (2,-2)  to  (2,-5);
		\draw[blue] (3,-3)  to  (3,-5);
		\draw[blue] (4,-4) to (4,-5);
		\node at (-4.5, -3.55) [fill = white, right, inner sep = 3pt] {{$\MMt$} } ;
	\end{tikzpicture} 
\end{subfigure}
%\hfill
\hspace{7mm}
\begin{subfigure}[b]{0.45\textwidth}
\begin{tikzpicture}[scale=0.6]
\path[step=1.0,black,thin,xshift=0.5cm,yshift=0.5cm] (-6.5,-6) grid (6,6);

		\draw[style=mygrid] (-5.0,-5.0) grid (5.0,5.0);
		\path[style=myfill] (0,0) -- (-5,5) -- (0,5) -- (0,0);
		\draw[->, thin, gray!80] (-5.0,0) -- (5.3,0) ;
		\draw (5.5,0) node[black] {$i$};
		\draw[->, thin, gray!80] (0,-5.0) -- (0,5.3);
		\draw (0,5.5) node[black] {$p\cd \sigma$};

		\foreach \y in {-5,...,5}{
		\foreach \x in {-\y,...,5}{
		\fill (\x,\y) circle (2.5pt);
	}
}

		\draw (-1,1) node[left] {\small $u$};
		\draw (-2,2) node[left] {\small $u^2$};
		\draw (-3,3) node[left] {\small $u^3$};
		\draw (-4,4) node[left] {\small $u^4$};
		\draw (-5,5) node[left] {\small $u^5$};

		\draw (0,0) node[left] {\small$1$};
		\draw (0,1) node[right] {\small$a$};
		\draw (0,2) node[right] {\small$a^2$};
		\draw (0,3) node[right] {\small$a^3$};
		\draw (0,4) node[right] {\small$a^4$};
		\draw (0,5) node[right] {\small$a^5$};
		\draw (1,-1) node[left] {\small$u^{-1}$};
		\draw (2,-2) node[left] {\small$u^{-2}$};
		\draw (3,-3) node[left] {\small$u^{-3}$};
		\draw (4,-4) node[left] {\small$u^{-4}$};
		\draw (5,-5) node[left] {\small$u^{-5}$};

		\draw[ red] (5, -5) -- (-5,5);
		\draw[ red] (5, -4) -- (-4,5);
		\draw[ red] (5, -3) -- (-3,5);
		\draw[ red] (5, -2) -- (-2,5);
		\draw[red] (5,-1) -- (-1,5);
\draw[red] (5,0) -- (0,5);
\draw[red] (5,1) -- (1,5);
\draw[red] (5,2) -- (2,5);
\draw[red] (5,3) -- (3,5);
\draw[red] (5,4) -- (4,5);

		\draw[blue] (5,-5)  to  (5,5);
		\draw[blue] (4,-4)  to  (4,5);
		\draw[blue] (3,-3)  to  (3,5);
		\draw[blue] (2,-2)  to  (2,5);
		\draw[blue] (1,-1)  to  (1,5);		
		\draw[blue] (0,0)  to  (0,5);
		\draw[blue] (0,0)  to  (0,5);
		\draw[blue] (-1,1)  to  (-1,5);
		\draw[blue] (-2,2)  to  (-2,5); 
		\draw[blue] (-3,3)  to  (-3,5);
		\draw[blue] (-4,4) to (-4,5);
		\node at (-4.5, -3.55) [fill = white, right, inner sep = 3pt] {{$H^{i+p\sigma}_{\Br}(\Es C_2, \Z/2)$} } ;
	\end{tikzpicture}
\end{subfigure}
\\
\begin{subfigure}[b]{0.45\textwidth}
\begin{tikzpicture}[scale=0.6]
\path[step=1.0,black,thin,xshift=0.5cm,yshift=0.5cm] (-6.5,-6) grid (6,6);
	\path[->,very thick] (0,6.3) edge (0,7.5);
		\draw[style=mygrid] (-5.0,-5.0) grid (5.0,5.0);
	\path[style=myfill] (2,-2) -- (5,-5) -- (2,-5) -- (2,-2);
		\draw[->, thin, gray!80] (-5.0,0) -- (5.3,0) ;
		\draw (5.5,0) node[black] {$i$};
		\draw[->, thin, gray!80] (0,-5.0) -- (0,5.3);
		\draw (0,5.5) node[black] {$p\cd \sigma$};

		\foreach \x in {2,...,5}{
		\foreach \y in {-5,...,5}{
		\fill (\x,\y) circle (2.5pt);
	}
}

		\draw[ red] (5, -5) -- (2,-2);
		\draw[ red] (5, -4) -- (2,-1);
		\draw[ red] (5, -3) -- (2,0);
		\draw[ red] (5, -2) -- (2,1);
		\draw[red] (5,-1) -- (2,2);
\draw[red] (5,0) -- (2,3);
\draw[red] (5,1) -- (2,4);
\draw[red] (4,-5) -- (2,-3);
\draw[red] (3,-5) -- (2,-4);
\draw[red] (5,2) -- (2,5);
\draw[red] (5,3) -- (3,5);
\draw[red] (5,4) -- (4,5);

			\draw[blue] (2,-5)  to  (2,5);
\draw[blue] (3,-5)  to  (3,5);
\draw[blue] (4,-5)  to  (4,5);
\draw[blue] (5,-5)  to  (5,5);

\node at (2,5) [left] {\small$\theta a^7$};
\node at (2,4) [left] {\small$\theta a^6$};
\node at (2,3) [left] {\small$\theta a^5$};
\node at (2,2) [left] {\small$\theta a^4$};
\node at (2,1) [left] {\small$\theta a^3$};
\node at (2,0) [left] {\small$\theta a^2$};
\node at (2,-1) [left] {\small$\theta a$};
\node at (2,-2) [left] {\small$\theta$};
		\node at (2,-3) [left] {\small$\frac{\theta}{a}$};
\node at (2,-4) [left] {\small$\frac{\theta}{a^2}$};
\node at (2,-5) [left] {\small$\frac{\theta}{a^3}$};

\draw (3.05,-3) node[right] {\small$\frac{\theta}{u}$};
\draw (4.05,-4) node[right]{\small $\frac{\theta}{u^2}$};
\draw (5.05,-5) node[right]{\small $\frac{\theta}{u^3}$};
		\node at (-4.5, -3.55) [fill = white, right, inner sep = 3pt] {{$\rH^{i+p\sigma}_{\Br}(\wt\Es C_2, \Z/2)$} } ;
	\end{tikzpicture}
\end{subfigure}
%\hfill
\hspace{7mm}
\begin{subfigure}[b]{0.45\textwidth}
	\begin{tikzpicture}[scale=0.6]
		\path[step=1.0,black,thin,xshift=0.5cm,yshift=0.5cm] (-6.5,-6) grid (6,6);
		
		\draw[style=mygrid] (-5.0,-5.0) grid (5.0,5.0);
		\draw[->, thin, gray!80] (-5.0,0) -- (5.3,0) ;
		\draw (5.5,0) node[black] {$i$};
		\draw[->, thin, gray!80] (0,-5.0) -- (0,5.3);
		\draw (0,5.5) node[black] {$p\cd \sigma$};
		
		\foreach \x in {2,...,3}{
			\foreach \y in {-5,...,5}{
				\fill (\x,\y) circle (2.5pt);
			}
		}
		
%		\draw[ red] (3, -5) -- (2,-4);
%		\draw[ red] (4, -5) -- (2,-3);
%		\draw[ red] (4, -4) -- (2,-2);
%		\draw[ red] (4, -3) -- (2,-1);
%		\draw[ red] (4, -2) -- (2,0);
%		\draw[red] (4,-1) -- (2,1);
%		\draw[red] (4,0) -- (2,2);
%		\draw[red] (4,1) -- (2,3);
%		\draw[red] (4,2) -- (2,4);
%		\draw[red] (4,3) -- (2,5);
%		\draw[red] (4,4) -- (3,5);
		\draw[ red] (3, -5) -- (2,-4);
%		\draw[ red] (4, -5) -- (2,-3);
		\draw[ red] (3, -4) -- (2,-3);
		\draw[ red] (3, -3) -- (2,-2);
		\draw[ red] (3, -2) -- (2,-1);
		\draw[red] (3,-1) -- (2,0);
		\draw[red] (3,0) -- (2,1);
		\draw[red] (3,1) -- (2,2);
		\draw[red] (3,2) -- (2,3);
		\draw[red] (3,3) -- (2,4);
		\draw[red] (3,4) -- (2,5);

		\draw[blue] (2,-5)  to  (2,5);
		\draw[blue] (3,-5)  to  (3,5);
%		\draw[blue] (4,-5)  to  (4,5);

		\node at (2,5) [left] {\small$\theta a^7$};
		\node at (2,4) [left] {\small$\theta a^6$};
		\node at (2,3) [left] {\small$\theta a^5$};
		\node at (2,2) [left] {\small$\theta a^4$};
		\node at (2,1) [left] {\small$\theta a^3$};
		\node at (2,0) [left] {\small$\theta a^2$};
		\node at (2,-1) [left] {\small$\theta a$};
		\node at (2,-2) [left] {\small$\theta$};
		\node at (2,-3) [left] {\small$\frac{\theta}{a}$};
		\node at (2,-4) [left] {\small$\frac{\theta}{a^2}$};
		\node at (2,-5) [left] {\small$\frac{\theta}{a^3}$};
		
		\draw (3.05,-3) node[right] {\small$\frac{\theta}{u}$};
%		\draw (4.05,-4) node[right]{\small $\frac{\theta}{u^2}$};
		\node at (-4.5, -3.55) [fill = white, right, inner sep = 3pt] {$\block_m,\,\,m=1$ } ;
	\end{tikzpicture}
\end{subfigure}
	\caption{$\bullet$ denotes a copy of $\Z/2$. Shading indicates the assembly of $\MMt$ from (parts of) $H^{\star}_{\Br}(\Es C_2)$ and $\rH^\star(\wt\Es C_2)$. }
	\label{point2}
\end{figure}

\begin{remark}\label{rem:blocks}
	It will be convenient to have  notation for certain submodules of 
$\rH^{\star}_{\Br}(\wt \Es C_2,\ul{\Z/2})$. For $i\geq 1$, write
	\[
	\block_i:= \Sigma^{2-2\sigma} \Z/2[ a^{\pm 1}, u^{-1}]/(u^{-2i}).
	\]
	This is a $\Z/2[a,u]$-module and hence an $\MMt$-module (elements of $\NC$ act by $0$).  Note that there is an identification
	\[
	\block_i \cong \bigoplus_{a\leq 2i+1}\rH^{a+*\sigma}_{\Br}(\wt\Es C_2,\ul{\Z/2}).
	\]
	For $i\leq 0$ we  set $B_i=0$. There are canonical $\MMt$-module quotients $\block_{i+1}\to \block_{i}$. Moreover, there are also  $\MMt$-module inclusions 
	\begin{equation}\label{eqn:incl}
	\block_i\hookrightarrow \block_{i+1}
	\end{equation}
	defined by the assignment $\frac{\theta}{u^n}a^m\mapsto\frac{\theta}{u^n}a^m$. 
	Composing  the  inclusion $\block_i\ \hookrightarrow \Sigma^{2-2\sigma} \Z/2[ a^{\pm 1}, u^{-1}]$ with \eqref{eqn:wtEtopt}, yields the $\MMt$-module map
	\begin{equation}\label{eqn:BiM}
		\block_i\to \MMt.
	\end{equation}
	Lastly we note that there are $\MMt$-module maps 
	\begin{equation}\label{eqn:byu}
	\cd u^2:\block_{i+1}\to \block_{i}, 
	\end{equation}
	defined by composing the quotient map with 
	multiplication by $u^2$, 
	\[
	\Sigma^{2-2\sigma}\Z/2[u^{-1}, a^{\pm 1}]/u^{-2i + 2} \xrightarrow{\cd u^2}
	\Sigma^{2-2\sigma}\Z/2[u^{-1}, a^{\pm 1}]/u^{-2i}.
	\]
	Explicitly it is the map  
	\[
	\frac{\theta}{u^n}a^m\mapsto \begin{cases} 
		\displaystyle{\frac{\theta}{u^{n-2}}a^m } & n\geq 2 \vspace{1mm}\\
		0 & \textrm{else}.
	\end{cases}
	\]
\end{remark}

\begin{figure}
	\begin{subfigure}[b]{0.45\textwidth}
		\begin{tikzpicture}[scale=0.6]
			\path[->,  thick] (9.5,1) edge node[midway, above]{$\cd u^2$}(11,1);
			\path[left hook->, thick] (11,-1) edge (9.5,-1);
			\draw[style=mygrid] (-1.0,-4.0) grid (9.0,4.0);
			\path[style=myfill] (1.8,-4.2) -- (1.8, 4.2) -- (3.2,4.2) -- (3.2,-4.2) -- (1.8,-4.2);
			\draw[->, thin, gray!80] (-1.0,0) -- (9.3,0) ;
			\draw[->, thin, gray!80] (0,-4.0) -- (0,4.3);
			
			\foreach \x in {2,...,8}{
				\foreach \y in {-4,...,4}{
					\fill (\x,\y) circle (2.5pt);
				}
			}
			\foreach \x in {2,...,7}{
			\foreach \y in {-3,...,4}{
				\path[red] (\x,\y) edge (\x+1, \y-1);
			}	
		}
			\foreach \x in {2,...,8}{
					\path[blue] (\x,-4) edge (\x, 4);
			}

			\node at (2,4) [left] {\small$\theta a^6$};
			\node at (2,3) [left] {\small$\theta a^5$};
			\node at (2,2) [left] {\small$\theta a^4$};
			\node at (2,1) [left] {\small$\theta a^3$};
			\node at (2,0) [left] {\small$\theta a^2$};
			\node at (2,-1) [left] {\small$\theta a$};
			\node at (2,-2) [left] {\small$\theta$};
			\node at (2,-3) [left] {\small$\frac{\theta}{a}$};
			\node at (2,-4) [left] {\small$\frac{\theta}{a^2}$};

			\draw (3.05,-3) node[right] {\small$\frac{\theta}{u}$};
			\draw (4.05,-4) node[right]{\small $\frac{\theta}{u^2}$};
			\node at (-0.9, -2.6) [fill = white, anchor=west] {$\block_{m+1}$ } ;
		\end{tikzpicture}
		\end{subfigure}
	\hspace{2mm}
		\begin{subfigure}[b]{0.45\textwidth}
		\begin{tikzpicture}[scale=0.6]

			\draw[style=mygrid] (-1.0,-4.0) grid (9.0,4.0);
			\draw[->, thin, gray!80] (-1.0,0) -- (9.3,0) ;
			\draw[->, thin, gray!80] (0,-4.0) -- (0,4.3);

			\foreach \x in {2,...,6}{
				\foreach \y in {-4,...,4}{
					\fill (\x,\y) circle (2.5pt);
				}
			}
			\foreach \x in {2,...,5}{
				\foreach \y in {-3,...,4}{
					\path[red] (\x,\y) edge (\x+1, \y-1);
				}	
			}
			\foreach \x in {2,...,6}{
				\path[blue] (\x,-4) edge (\x, 4);
			}

			\node at (2,4) [left] {\small$\theta a^6$};
			\node at (2,3) [left] {\small$\theta a^5$};
			\node at (2,2) [left] {\small$\theta a^4$};
			\node at (2,1) [left] {\small$\theta a^3$};
			\node at (2,0) [left] {\small$\theta a^2$};
			\node at (2,-1) [left] {\small$\theta a$};
			\node at (2,-2) [left] {\small$\theta$};
			\node at (2,-3) [left] {\small$\frac{\theta}{a}$};
			\node at (2,-4) [left] {\small$\frac{\theta}{a^2}$};

			\draw (3.05,-3) node[right] {\small$\frac{\theta}{u}$};
			\draw (4.05,-4) node[right]{\small $\frac{\theta}{u^2}$};
			
			\node at (-0.9, -2.6) [fill = white, anchor=west] {$\block_m$ } ;
			
		\end{tikzpicture}
	\end{subfigure}
	\caption{The modules $\block_i$ and the maps \eqref{eqn:incl} and \eqref{eqn:byu} between them. The shaded region indicates the kernel of $\cd u^2:\block_{m+1}\to \block_{m}$. }
	\label{fig:Bmaps}
\end{figure}

\subsection{Our computation}
 We describe the main computations.  
 Bredon motivic cohomology is graded by a $4$-tuple of integers, written as $(a+p\sigma, b+q\sigma)$; 
 this $4$-tuple is viewed as a pair of $C_2$-representations (here $\sigma$ denotes the sign representation), the first one is the \emph{cohomological degree} and the second representation is the \emph{weight}.  The grading by $4$-tuples presents an organizational problem.  Our solution is to  organize Bredon motivic cohomology into $\MMt$-modules, which we now explain. 
 If $X$ is a complex variety with $C_2$-action, 
Betti realization induces a comparison homomorphism
 \[
 \Re:H^{a+p\sigma, b+q\sigma}_{C_2}(X, \ul{\Z/2}) \to H^{a+p\sigma}_{\Br}(X(\C), \ul{\Z/2})
 \]
 between Bredon motivic cohomology of $X$ and the Bredon cohomology of the $C_2$-topological space $X(\C)$. When $X=\spec(\C)$, this induces an isomorphism of bigraded rings by \aref{prop:ptiso},
 \[
 H^{\star, 0}_{C_2}(\C, \ul{\Z/2}) \xrightarrow{\cong}
 \MMt.
 \]
 In particular, we can view $H^{\star,b+q\sigma}_{C_2}(X,\ul{\Z/2})$ as an $\MMt$-module, for each $b,q$.

The free motivic $C_2$-space $\EG C_2$ can be modeled as $\A(\infty \sigma)\setminus 0$, where $\A(n\sigma)$ is the $n$-dimensional sign representation. There is a motivic isotropy separation sequence
$\EG C_{2+}\rightarrow S^0\rightarrow \EGt C _2$, where $\EGt C_2$ is defined so that this a cofiber sequence (see \aref{sub:equivariant} for details), which breaks the problem of computing $H^{\star,\star}_{C_2}(\C,\ul{\Z/2})$ into pieces.

Each of $\EG C_2$ and $\wt \EG C_2$ determine a region of   $H^{\star,\star}_{C_2}(\C, \ul{\Z/2})$ and Betti realization determines the remaining nonzero region, see \aref{thm:additive}. These regions are shown in \aref{TEST1}. In this picture we have projected onto the plane determined by the weight. In particular, the displayed elements do not all live the same cohomological degree.

\begin{figure}[H]
	\begin{tikzpicture}[scale=0.7]
		\draw[->, thin, gray!40] (-6.0,0) -- (6.3,0) node[right, black] {$b$};
		\draw[->, thin,  gray!40] (0,-6.0) -- (0,6.3) node[above, black] {$q\cd\sigma$};
		\path[style = myfill1] (0,0) -- (-6,6) -- (0,6) -- (0,0);
		\path[style = myfill2] (0,0) -- (0,6) -- (6,6) -- (6,-6) -- (0,0);
		\path[style = myfill3] (1,-1) -- (1,-6) -- (6,-6) -- (1,-1);
		\draw[ultra thick] (0,0) -- (-6,6);
		\draw[ultra thick] (0,0) -- (0,6);
		\draw[ultra thick] (0,0) -- (6,-6);
		\draw[ultra thick] (1,-1) -- (1,-6); 
		\node at (-2.7,3) [right] {$H^{\star,\star}_{C_2}(\EG C_2)$};
		\node at (1.3,3) [right] {$H^{\star}_{\Br}(\pt)$};
		\node at (1.3,-4.3) [right] {$\rH^{\star,\star}_{C_2}(\wt\EG C_2)$};
		\fill (0,1) circle(3pt) node[above right] {$\tau_\sigma$};
		\fill (0,0) circle(3pt) node[above right] {$1$};
		\fill (-1,1) circle(3pt) node[below left] {$\xi$};
		\fill (1,-1) circle(3pt) node[above right] {$\mu$};
		
	\end{tikzpicture} \hfill 
	\caption{Regions of {$H^{\star, b+q\sigma}_{C_2}(\C,\Z/2)$} determined by $\EG C_2$, Betti realization, and $\wt\EG C_2$. The degrees of the displayed elements are $|\xi|=(-2+2\sigma, -1+\sigma)$, $|\mu| = (0,1-\sigma)$, $|\tau_\sigma| = (0,\sigma)$.   }
	\label{TEST1}
\end{figure}
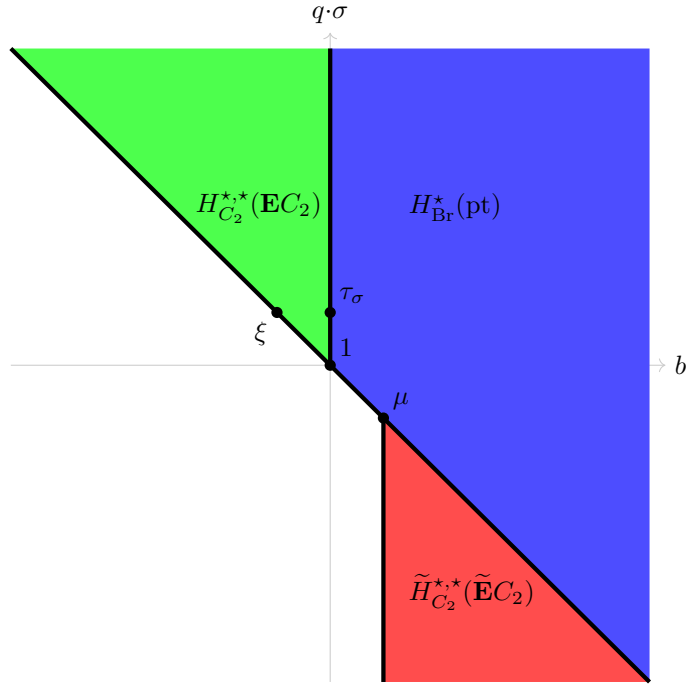

In integer bidegrees, the Bredon motivic cohomology of $\EG C_2$ agrees with ordinary motivic cohomology of 
$\EG C_2/C_2 = \BG C_2$. 
 The motivic cohomology of $\BG C_2$ was computed by Voevodsky \cite[Theorem 6.10]{Voc}. In our case, where the base field is $\C$, his computation takes the form
 \begin{equation}\label{eqn:voebg}
 H^{*,*}(\BG C _2,\Z/2)\cong \Z/2[\tau][e_1,e_2]/(e_1^2=e_2\tau),
 \end{equation}
  where $|e_1| = (1,1)$, $|e_2| = (2,1)$, and $|\tau| = (0,1)$. 
In \aref{5},  we leverage Voevodsky's computation, Betti realization, and that the cohomology of 
$\EG C_2$ is $(-2+2\sigma, -1+\sigma)$-periodic, to find an equivariant lift of $\tau$ to an element 
$\tau_\sigma\in H^{0,\sigma}_{C_2}(\EG C_2)$ 
such that multiplication by $\tau_\sigma$ is an isomorphism  whenever $b+q\geq 0$. Thus we find that 
	\[
  	H^{\star,\star}_{C_2}(\EG C_2, \ul{\Z/2})  {\cong} 
  	\MMt[\xi^{\pm 1}, \tau_\sigma],
  	\]
  	where $|\xi| = (-2+2\sigma, -1+\sigma)$.

The cohomology of $\wt\EG C_2$ is both $(\sigma,0)$ and $(0,\sigma)$-periodic and Betti realization identifies  
$\rH^{\star, b+q\sigma}_{C_2}(\wt\EG C_2)$  with the sub-$\MMt$-module of $\rH^{\star}_{\Br}(\wt\Es C_2)$ 
\[
\rH^{\star, b+q\sigma}_{C_2}(\wt\EG C_2,\ul{\Z/2})\cong \block_{b},
\]
where $\block_{b}=  \Sigma^{2-2\sigma} \Z/2[ a^{\pm 1}, u^{-1}]/(u^{-2b})$ (if $b\geq 1$, it is $0$ otherwise) is the $\MMt$-module introduced in \aref{rem:blocks}. We keep track of weights via the elements $\tau_\sigma$ and $\mu$, where 
$|\tau_\sigma| = (0,\sigma)$ and 
$|\mu|= (0,1-\sigma)$, so that
\[
\rH^{\star,\star}_{C_2}(\wt\EG C_2,\ul{\Z/2}) \cong \bigoplus_{i\geq 1, j\in \Z}\hspace{-0.8em}\block_i\hspace{-1pt}\left\{\mu^i\tau_\sigma^j\right\}.
\]

Having determined the $\MMt$-module structures in all of the regions in \aref{TEST1}, we determine the multiplicative structure in  \aref{thm:ptring}, where we find there is an isomorphism of $\MMt$-algebras 
	\begin{equation}\label{eqn:H(C)i}
 H_{C_2}^{\star,\star}(\C,\ul{\Z/2}) \cong	\MMt[\xi,\tau_\sigma, \mu]/(\xi\mu-u^2)\oplus\Big(\bigoplus_{i,j\geq 1}\hspace{-0.1em}\block_i\hspace{-1pt}\left\{\frac{\mu^i}{\tau_\sigma^j}\right\}\Big).
	\end{equation}
	The left hand summand comes from the regions determined by $\EG C_2$ and Betti realization in \aref{TEST1}. The right hand summand arises from the region determined by $\wt\EG C_2$. The multiplicative structure involving elements in this region is determined as follows.
	\begin{enumerate}[(i)]
		\item $\cd\mu:\block_{i}\left\{\frac{\mu^i}{\tau_\sigma^j}\right\}\to \block_{i+1}\left\{\frac{\mu^{i+1}}{\tau_\sigma^j}\right\}$ is the inclusion \eqref{eqn:incl}.
		\item $\cd\xi: \block_{i}\{\frac{\mu^i}{\tau_\sigma^j}\}\to \block_{i-1}\left\{\frac{\mu^{i-1}}{\tau_\sigma^j}\right\}$ is \eqref{eqn:byu}, multiplication by $u^2$.

		\item  $\cd\tau_\sigma: \block_{i}\left\{\frac{\mu^i}{\tau_\sigma^{j}}\right\}\to\block_{i}\left\{\frac{\mu^i}{\tau_\sigma^{j-1}}\right\}$ is the identity if $j\geq 2$.
		\item a $\tau_\sigma$-multiplication starting in
		weights $i-(i+1)\sigma$, is the map $\cd\tau_\sigma:B_i\{\frac{\mu^i}{\tau_\sigma}\} \to \MMt\{\mu^i\}$ determined by \eqref{eqn:BiM}. These are exactly the multiplications crossing the border from the region determined by $\wt\EG C_2$ into the region determined by Betti realization, see \aref{TEST1}.

		\item All products in the right hand summand are trivial.
	\end{enumerate} 
\vspace{2mm}

\noindent {\bf Outline.}
A brief outline of the paper is as follows. Sections 1 and 2 are devoted to the introduction and preliminaries. The main computations of Bredon motivic cohomology are carried out in Sections 3 and 4.  In the last section,  we generalize the results to any algebraically closed field of characteristic zero via a rigidity result for Bredon motivic cohomology.
  \vspace{1mm}\\
 \noindent {\bf{Notation.}} 
\begin{itemize}[label = {$\cdot$}]
\item $H^{a+p\sigma,b+q\sigma} _{C _2}(X,\ul{A})$ is the Bredon motivic cohomology of a $C _2$-smooth scheme, with coefficients $A$. 
\item $H^{n,q}(X,A)$ is motivic cohomology of a smooth scheme $X$.
\item $H^{a+b\sigma}_{\Br}(X,\ul{A})$ is the Bredon cohomology of a $C_2$-topological space $X$ with coefficients in the constant Mackey functor $\underline{A}$.

\item We write $\star$ for an ${\rm RO}(C_2)$-grading and $*$ for a $\Z$-grading. \\For example, $H^{\star,\star} _{C _2}(X)=\oplus_{a,b,p,q}  H^{a+p\sigma,b+q\sigma} _{C _2}(X)$ and $H^{*,*}(X)=\oplus_{a,b} H^{a,b}(X)$.
  \item  $S^\sigma$ is the topological sphere associated to the real sign representation $\sigma$.
\item  All $C_2$-varieties are over $\C$ and we view $C_2$ as a group scheme by $C_2=\spec(\C)\sqcup \spec(\C)$.
\item $\MMtn:=H^{\star}_{\Br}(\pt, \ul{\Z/n})$.
 
  \end{itemize}
\vspace{1mm}

  \noindent {\bf{Acknowledgements.}} The authors wish to thank Institut Mittag-Leffler, Stockholm, where the research of this paper started in 2017 during the program on Algebro-Geometric and Homotopical Methods. Heller and {\O}stv{\ae}r thank the Isaac Newton Institute for Mathematical Sciences for support and hospitality during the program on K-theory, algebraic cycles and motivic homotopy theory in 2020. We are grateful to the referee for useful comments on a previous draft of this paper.

 \section{Preliminaries}
 We record some background on Bredon motivic  cohomology.

 \subsection{Equivariant motivic homotopy}\label{sub:equivariant}
 The stable equivariant motivic homotopy category $\SH^{C_2}(k)$ is the stabilization of Voevodsky's category of equivariant motivic spaces  \cite{Del}, with respect to Thom spaces of representations. We recall a few key facts and the notation we use in the case $G=C_2$. See \cite{Hoy}, \cite{HOV1}, or \cite{GH} for details.

 Let $V=a+p\sigma$ be  a $C_2$-representation, where $a$ denotes the $a$-dimensional trivial representation and $p\sigma$ is the $p$-dimensional sign representation. 
 We write $\A(V)$ and $\P(V)$ for the  $C_2$-schemes $\A^{\dim(V)}$ and $\P^{\dim(V)-1}$ equipped with the corresponding action coming from $V$. The associated motivic representation sphere is 
 \[
 T^V:=\P(V\oplus 1)/\P(V).
 \] 
 Indexing is based on the following four spheres. There are two topological spheres $S^1$, $S^\sigma$ and two algebro-geometric spheres $S_t=(\A^1\setminus\{0\},1)$ equipped with trivial action, and $S^{\sigma}_t=(\A^1\setminus\{0\},1)$ equipped the 
$C_2$-action $x\rightarrow x^{-1}$.  We write
 \[
 S^{a+p\sigma,b+q\sigma}:=S^{a-b}\wedge S^{(p-q)\sigma}\wedge S^b _t\wedge S^{q\sigma}_t.
 \]
 In this indexing, we have $T\simeq S^{2,1}$ and $T^\sigma\simeq S^{2\sigma, \sigma}$. 
The stable equivariant motivic homotopy category
 $\SH^{C_2}(k)$ is the stabilization of (based) $C_2$-motivic spaces with respect to the motivic sphere 
 $T^{\rho}$ corresponding to the regular representation $\rho= 1 +\sigma$. 
 
 We make use of two fundamental cofiber sequences in 
 $\SH^{C_2}(k)$. The first is
 \begin{equation}\label{eqn:cof1}
 C_{2+}\to S^{0}\to S^{\sigma}.
 \end{equation}
 The second is
 \begin{equation}\label{eqn:cof2}
 \EG C_{2+}\rightarrow S^0\rightarrow \EGt C _2.
 \end{equation}
 Here, $\EG C_{2}$ is the universal free motivic $C_2$-space. It has a geometric model, $\EG C _2\simeq \colim _n\A(n\sigma)\setminus \{0\}$, see \cite[Section 3]{GH}.
The quotient  
$\EG C _2/C _2\simeq \colim _n \left(\A(n\sigma)\setminus \{0\}\right)/C _2$ is the geometric classifying space $\BG C_2$ constructed by Morel-Voevodsky \cite{MVo} and Totaro \cite{T}.  
 Note also that  
   $\EGt C _2=\colim _n S^{2n\sigma,n\sigma}$.
 In particular, the maps $S^0\to T^\sigma$ and $S^0\to S^\sigma$ induce equivalences
 \[
 	\EGt C_2 \xto{\simeq} T^{\sigma}\wedge\EGt C_2\,\,\,\textrm{  and  }\,\,\,
 	\EGt C_2 \xto{\simeq} S^{\sigma}\wedge\EGt C_2,
 \]
  see \cite[Proposition 2.9]{HOV1}.

 Equipping a variety with trivial action $\Sm_{k}\to \Sm^{C_2}_k$ induces 
a functor $\SH(k)\to \SH^{C_2}(k)$.
 
 \subsection{Bredon motivic cohomology}

 Bredon motivic cohomology is represented in $\SH^{C_2}(k)$ by  
 the spectrum $M\ul{A}$ associated to an abelian group $A$, where 
 $M\ul{A} _n=A _{tr,C _2}(T^{n\rho})$ is the free presheaf with equivariant transfers, see \cite{HOV1} for details.

  \begin{definition}[\cite{HOV1}] The Bredon motivic cohomology of a motivic $C _2$-spectrum $E$ with coefficients in an abelian group $A$ is defined by 
  \[
  \rH^{a+p\sigma,b+q\sigma} _{C_2}(E,\ul{A})=[E,S^{a+p\sigma,b+q\sigma}\wedge M\ul{A}] _{\SH^{C _2}(k)}.
  \]
  \end{definition}

If $X\in \Sm^{C_2}_k$ we typically write
\[
H^{a+p\sigma,b+q\sigma}_{C_2}(X, \ul{A}) := 
\rH^{a+p\sigma,b+q\sigma}_{C_2}(X_+, \ul{A}).
\]
 When $A$ is a ring, then $H^{\star,\star} _{C _2}(X, \ul{A})$ is a graded commutative ring by \cite[Proposition 3.24]{HOV1}. Specifically this means that if $x\in H^{a+p\sigma, b+q\sigma} _{C _2}(X, \ul{A})$ and $y\in H^{c+s\sigma, d+t\sigma} _{C _2}(X, \ul{A})$, then
\[
x\cup y = (-1)^{ac+ps}y\cup x.
\]

A few features of this theory, which we use are the following (see \cite{HOV1}).
\begin{itemize}[label = {$\cdot$}]
\item If  $E$ is in the image of $\SH(k)\to \SH^{C_2}(k)$, i.e. it has ``trivial action", then  there is an isomorphism in integral bidegrees with ordinary motivic cohomology,
\[
\rH^{a,b}_{C _2}(E, \ul{A}) \cong \rH^{a,b}(E, A).
\]

\item If $X$ has free action, then there is an isomorphism in integral bidegrees with ordinary motivic cohomology,
\[
H^{a,b} _{C _2}(X,\ul{A})\cong H^{a,b}(X/C_2, A).
\]
\item  $H^{\star,\star}_{C_2}(\EG C_2, \ul{A})$  is  $(-2+2\sigma,-1+\sigma)$-periodic.
\end{itemize}

\subsection{Betti realization}

  The map of sites 
  $\Sm _\C^{C_2}\rightarrow {\rm Top}^{C_2}$, given by $X\rightarrow X(\C)$, where the set of complex points is equipped with the analytic topology,  extends to a functor
  $\Re:\SH^{C _2}(\C)\rightarrow \SH^{C _2}$ between the stable equivariant motivic homotopy category over $\C$ and the classical stable equivariant homotopy category. We refer to this functor as the Betti realization.

   The indexing of the spheres above was chosen to interact well with complex Betti realization; we have 
$\Re (S^{a+p\sigma,b+q\sigma})\simeq S^{a+p\sigma}$.  
  
  By \cite[Theorem A.29]{HOV1}, $\Re(M\ul{A})  \simeq H\ul{A}$, where $H\ul{A}$ is the equivariant Eilenberg-MacLane spectrum associated to the constant Mackey functor $\ul{A}$. 
   In particular, for any smooth $C _2$-scheme over $\C$ there is a map
\[
\Re :H^{a+p\sigma,b+q\sigma} _{C _2}(X, \ul{A})\rightarrow H^{a+p\sigma}_{\Br}(X(\C),\underline{A}).
\]

  Betti realization takes the cofiber sequences \eqref{eqn:cof1} and \eqref{eqn:cof2}
  to the corresponding ones in $\SH^{C_2}$. \\
    Given $\XX$ in $\SH^{C_2}(\C)$,
     we obtain a comparison of long exact sequences
  \begin{equation}\label{eqn:les1}
  \begin{tikzcd}[column sep=small]
  \cdots \ar[r] & \rH^{a+(p-1)\sigma, b +q\sigma}_{C_2}(\XX, \ul{A})  \ar[r]\ar[d] & \rH^{a+p\sigma, b +q\sigma}_{C_2}(\XX, \ul{A}) \ar[r]\ar[d] & \rH^{a+p, b +q}_{}(\XX, A) \ar[r]\ar[d] & \cdots \\
  \cdots \ar[r] & \rH^{a+(p-1)\sigma}_{\Br}(\Re(\XX),\ul A) \ar[r]& \rH^{a+p\sigma}_{\Br}(\Re(\XX),\ul A) \ar[r] &   \rH^{a+p}_{\sing}(\Re(\XX), A) \ar[r] & \cdots 
  \end{tikzcd}
  \end{equation}
  as follows. The top row of this sequence is obtained by smashing $\XX$ with \eqref{eqn:cof1} and the bottom is obtained similarly via Betti realization. Here we use the identifications 
  $\rH^{a+p\sigma, b +q\sigma}_{C_2}(C_{2+}\wedge\XX)\cong \rH^{a+p, b +q}_{}(\XX)$ and $\rH^{a+p\sigma}_{\Br}(\Re(\XX)) \cong \rH^{a+p}_{\sing}(\Re(\XX))$ and the $\Re$ is compatible with these identifications, see \cite[Proposition 3.14]{HOV1}.
  Smashing $\XX$ with \eqref{eqn:cof2} we obtain the comparison of long exact sequences
  \begin{equation}\label{eqn:les2}
  \begin{tikzcd}[column sep=small]
  \cdots \ar[r] & \rH^{a+p\sigma, b +q\sigma}_{C_2}(\EGt C_2 \wedge \XX, \ul{A}) \ar[r]\ar[d] & \rH^{a+p\sigma, b +q\sigma}_{C_2}(\XX, \ul{A}) \ar[r]\ar[d] & \rH^{a+p\sigma, b +q\sigma}_{C_2}(\EG C_{2+}\wedge\XX, \ul{A}) \ar[r]\ar[d] & \cdots 
  \\
  \cdots \ar[r] & \rH^{a+p\sigma}_{\Br}(\wt\E C_2\wedge \Re(\XX), \ul{A}) \ar[r]& \rH^{a+p\sigma}_{\Br}(\Re(\XX),\ul{A}) \ar[r] & 
  \rH^{a+p\sigma}_{\Br}(\E C_{2+}\wedge \Re(\XX),\ul{A}) \ar[r] & \cdots .
  \end{tikzcd}
  \end{equation}

Using the Beilinson-Lichtenbaum theorem proved by Voevodsky and Rost \cite{Voev:miln},\cite{Voev:BK}, it is shown in \cite{HOV1} that Betti realization is an isomorphism in a suitable range, on Bredon cohomology of smooth schemes.  
In \aref{sec:BC}, we will see that a stronger result holds for $X=\spec(\C)$. For the moment, we note that in nonnegative integer weights, we always have an isomorphism for finite coefficients. 
In particular, Betti realization induces an isomorphism of rings 
$H^{\star,0}_{C _2}(\C,\ul{\Z/n})\cong \MMtn$ and 
so $H^{\star,\star}_{C_2}(X,\ul{\Z/n})$ is a module over $\MMtn$. 
In fact, by \cite{Voi} or the same argument below,  Betti realization is an isomorphism in weight zero even with $\Z$-coefficients.

\begin{proposition}\label{prop:ptiso}
Let $A$ be a finite abelian group and $b\geq 0$. Betti realization induces an isomorphism for any $a,p$
\[
H^{a+p\sigma, b}_{C_2}(\C, \ul{A})\xrightarrow{\cong} H^{a+p\sigma}_{\Br}(\pt,\ul{A}).
\]
\end{proposition}
\begin{proof}
If $b\geq 0$, then $H^{a,b}(\C, A)\to H^{a}_{\sing}(\C, A)$ is an isomorphism for all $a$. In particular the result holds for $p=0$.  Using the comparison long exact sequence \eqref{eqn:les1} and the five lemma, the result holds for all $p$ by induction. 
\end{proof}

 \subsection{Vanishing of  Bredon motivic cohomology}

An important feature of motivic cohomology is its vanishing regions. If $X\in \Sm_k$  then 
$H^{a,b}(X,\Z/n) = 0$ in any of the following cases
\begin{enumerate}
	\item $a>2b$, 
	\item $a> b+\dim(X)$, or
	\item $b<0$.
\end{enumerate}

The vanishing regions for $H^{\star,\star}_{C_2}$ are more complicated. In this subsection  $k$ is a field with ${\rm char}(k)\neq 2$.

 \begin{proposition}\label{crt} 
	\;
	\begin{enumerate}
		\item $\rH^{*,*} _{C _2}(\EGt C _2,\ul{\Z/n})\cong \rH^{*,*}(\Sigma \BG C _2,\Z/n)$.
		\item $\rH^{\ast,b+q\sigma}_{C _2}(\EGt C _2,\ul{\Z/n})= 0$ if $b\leq 0$
	\end{enumerate}
	
\end{proposition}
\begin{proof} 
	Since $\rH^{\ast,\ast}_{C _2}(\EGt C _2,\ul{\Z/n})$ is $(\sigma,0)$ and $(0,\sigma)$-periodic, (2) follows from (1).
	The first statement follows from the long exact sequence induced by \eqref{eqn:cof2}. Indeed, by 
	\cite[Proposition 3.16]{HOV1} the map  
	$H^{*,*}_{C_2}(k,\ul{\Z/n})\to H^{*,*}_{C_2}(\EG C_2, \ul{\Z/n})$ is isomorphic to the split monomorphism
	$f^*:H^{*,*}(k,\Z/n)\to H^{*,*}(\BG C_2,\Z/n)$, where $f:\BG C_2 \to \spec(k)$ is the projection. Thus $\rH^{*+1,*}_{C_2}(\EGt C_2,\ul{\Z/n})$ is isomorphic to the cokernel of $f^*$, which is $\rH^{*+1,*}(\Sigma \BG C_2,\Z/n)$.

\end{proof}

\begin{proposition} \label{van} 
	Let $X\in \Sm^{C_2}_k$ and suppose that $b+q<0$ and $b<0$.  Then
	\[
	H^{\star,b+q\sigma}_{C_2}(X,\ul{\Z/n})=0.
	\]
\end{proposition}
\begin{proof} 
	Since $b<0$, we have  
	$\rH^{a+p\sigma,b+q\sigma}_{C_2}(\EGt C _2\wedge X_+)=0$.  
	Using the cofiber sequence \eqref{eqn:cof2}, we 
	  obtain 
	  \[
	H^{a+p\sigma,b+q\sigma}_{C_2}(X)\xto{\cong} 
	H^{a+p\sigma, b+q\sigma}_{C_2}(\EG C _2\times X).
	  	\]
	  	 Since $H^{a+p\sigma, b+q\sigma}_{C_2}(\EG C _2\times X)\cong H^{a+2q+(p-2q)\sigma, b+q}_{C_2}(\EG C _2\times X)$, it suffices to see that  
	  	 $H^{a+p\sigma, n}_{C_2}(\EG C _2\times X)=0$ for $n<0$. This follows from the case $p=0$, by induction using \eqref{eqn:cof1}.

\end{proof}

\begin{proposition} \label{tra}
	If $a\geq 2b+2$ then for any $X\in \Sm^{C_2}_k$,
	\begin{enumerate}
		\item $\rH^{a+p\sigma,b+q\sigma} _{C _2}(X_+\wedge \EGt C _2,\ul{\Z/n})=0$, and
		\item the projection $X\times \EG C_2\to X$ induces an isomorphism 
\[H^{a+p\sigma,b+q\sigma}_{C _2}(X,\ul{\Z/n})\xrightarrow{\cong} 
H^{a+p\sigma,b+q\sigma} _{C _2}(X\times \EG C _2,\ul{\Z/n}).\]		
	\end{enumerate}
\end{proposition}
\begin{proof} 
	The two statements are equivalent using \eqref{eqn:cof2}. Therefore, we will establish the first one. 
Since	
$\rH^{a+p\sigma,b+q\sigma} _{C _2}(X_+\wedge \EGt C _2,\ul{\Z/n}) 
\cong \rH^{a,b} _{C _2}(X_+\wedge \EGt C _2,\ul{\Z/n})$, we can assume that $p=q=0$.  We can assume that $X$ has trivial action, since 
$\rH^{a,b} _{C _2}(X^{C_2}_+\wedge \EGt C _2,\ul{\Z/n})\cong 
\rH^{a,b} _{C _2}(X_+\wedge \EGt C _2,\ul{\Z/n})$ by \cite[Proposition 4.10]{GH}.

Consider the exact sequence from \eqref{eqn:cof2}.
	\[
	\cdots \to H^{a-1,b} _{C _2}(X\times \EG C _2,\ul{\Z/n}) \rightarrow 
	\rH^{a,b} _{C _2}(X_+\wedge \EGt C _2,\ul{\Z/n})\rightarrow
	 H^{a,b} _{C _2}(X,\ul{\Z/n})\rightarrow 
	 H^{a,b} _{C _2}(X\times \EG C _2, \ul{\Z/n})\to\cdots .
	 \]
	Now $H^{a,b} _{C _2}(X\times \EG C_2,\ul{\Z/n})\cong H^{a,b}(X\times_{C_2}\EG C_2,\ul{\Z/n})$, 
	 and if $a>2b$ this last group is zero and so the proposition follows. 
	
\end{proof}
\begin{proposition} \label{1van} 
	For any $X\in \Sm^{C_2}_k$, if 
	$a\geq 2b+2$ and $p\geq 2q$ then we have that  
	\[
	H^{a+p\sigma,b+q\sigma}_{C_2}(X,\ul{\Z/n})=0.
	\] 	
\end{proposition}
\begin{proof} 
  By \aref{tra}, it suffices to show that 
  $H^{a+p\sigma,b+q\sigma}_{C_2}(X\times \EG C_2, \ul{\Z/n})=0$.
  By \cite[Theorem 5.4]{HOV1}, 
  \[
  H^{a+p\sigma,b+q\sigma}_{C_2}(X\times \EG C_2)
  \cong H^{a+2q+(p-2q)\sigma,b+q} _{C _2}(X\times \EG C _2).
  \]
  If $p=2q$,  then $H^{a+2q+(p-2q)\sigma,b+q} _{C _2}(X\times \EG C _2) \cong H^{a+2q,b+q} (X\times_{C_2} \EG C _2)$. This group vanishes when $a>2b$. To conclude the proposition, we use 
  the long exact sequence obtained from \eqref{eqn:cof1} and induction on $p\geq 2q$.

\end{proof}

The following example shows that the general vanishing range in the previous proposition can't be improved.

\begin{example} 
	For any $p$, we have 
	\[
	H^{2+p\sigma, 1-2\sigma}_{C_2}(\P^1, \ul{\Z/n})\neq 0.
	\]
	To see this, we first note that 
	$H^{\star,1-2\sigma}_{C_2}(\P^1\times \EG C _2, \ul{\Z/n})=0$ (see the proof of \aref{neq}) and therefore from \eqref{eqn:cof2} we see that
	$\rH^{2+p\sigma,1-2\sigma}_{C_2}(\P^1_+\wedge \EGt C _2,\ul{\Z/n})\cong H^{2+p\sigma,1-2\sigma}_{C_2}(\P^1, \ul{\Z/n})$. Since $\P^1_+\simeq T\vee S^0$, we thus have
	\[
	\rH^{2+p\sigma,1-2\sigma}_{C_2}(\P^1 _+\wedge \EGt C _2, \ul{\Z/n})\cong  \rH^{1,1}_{C_2}(\EGt C _2,\ul{\Z/n})\oplus \rH^{2,1}_{C_2}(\EGt C _2, \ul{\Z/n}).
	\]
	This group is nonzero since $\rH^{*,*}_{C_2}(\EGt C_2,\ul{\Z/n}) \cong \rH^{*,*}(\Sigma\BG C_2,\Z/n)$.
\end{example}

\section{Bredon motivic cohomology of \texorpdfstring{$\EG C_2$}{EC2}}\label{5}
In this section, we compute the Bredon motivic cohomology ring of $\EG C_2$. 
Betti realization plays a key role in our determination of this ring and we start by leveraging  motivic cohomology of $\BG C_2$. 
From \cite[Proposition 3.16]{HOV1}, we have an isomorphism
   $H^{a,b} _{C _2}(\EG C _2,A)\cong H^{a,b}(\BG C _2,A)$ and this isomorphism fits into the commutative diagram.
   \[
   \begin{tikzcd}
   H^{a,b} _{C _2}(\EG C _2,\ul{A})\ar[d,"\Re"']\ar[r,"\cong"] & H^{a,b}(\BG C _2,A)\ar[d, "\Re"] \\
   H^{a} _{\Br}(\E C _2,\ul{A})\ar[r, "\cong"] & H^{a}_{\sing}({\rm B} C _2,A).
   \end{tikzcd}
   \]
   
\begin{lemma}\label{lema:leqb}
	Let $A$ be a finite abelian group. Betti realization
	\[
	\Re:H^{a,b}(\BG C_2, A)\to H^{a}_{\sing}({\rm B} C_2, A)
	\]
	is an isomorphism if $a\leq 2b$. 
\end{lemma}
\begin{proof}
If $A=\Z/2$, this can be read off of Voevodsky's computation \eqref{eqn:voebg}, since Betti realization of $e_1$ is the generator of $H^*_{\sing}({\rm B} C _2,\Z/2)$. In general, we use
that $\BG C_{2+}$ sits in the cofiber sequence, see \cite[Section 6]{Voc},
\[
\BG C_{2+}\to \P^{\infty}_+ \to {\rm Th}(\mcal{O}(-2)).
\]
The lemma follows  from the comparison of long exact sequences induced by this cofiber sequence, the five lemma, the Thom isomorphism, and that $\Re:H^{a,b}(\P^{\infty};A)\to H^{a}_{\sing}(\C P^{\infty}, A)$ is an isomorphism if $a\leq 2b$.
\end{proof}

\begin{proposition}\label{Bocal}
Suppose $b+q\geq 0$. Then
\[
H^{a+p\sigma,b+q\sigma}_{C_2}(\EG C _2, \ul{\Z/n})
\cong 
\begin{cases}
	H^{a+p\sigma}_{\Br}(\Es C_2,\ul{\Z/n}) & a\leq 2b \\
	0 & a = 2b+1 \\
	H^{a+2q+(p-2q)\sigma}_{\Br}(\pt, \ul{\Z/n}) & a\geq 2b+2.
\end{cases}
\]
Furthermore, Betti realization is an isomorphism if $a\leq 2b$. It is   multiplication by  $2$ if $a\geq 2b+2$, $p=-a$, and $a$ is even. All other Betti realizations are zero.

\end{proposition}
\begin{proof} 
		Since $H^{\star, \star}_{C_2}(\EG C_2, \ul{\Z/n})$ is $(-2+2\sigma, -1+\sigma)$ periodic by \cite[Theorem 5.4]{HOV1} and the statement of the proposition is compatible with this periodicity,
		it suffices to treat the case $q=0$. We now assume that $q=0$.
		
		When  $p=0$, then $H^{a,b}_{C_2}(\EG C_2, \ul{\Z/n})\cong H^{a,b}(\BG C_2, \Z/n)=0$ if $a>2b$ and 
		by \aref{lema:leqb}, Betti realization is an isomorphism $ a\leq 2b$.

We suppress the coefficient group  for typographical simplicity and proceed by induction on $p$. 
To begin with, we use the comparison of exact sequences \eqref{eqn:les1}
	\[
	\begin{tikzcd}[column sep=small]
	\cdots \ar[r] & H^{i+p, b}_{}(\C) \ar[r]\ar[d] & H^{i+p\sigma, b }_{C_2}(\EG C_2) \ar[r]\ar[d] & H^{i+(p+1)\sigma, b }_{C_2}(\EG C_2) \ar[r]\ar[d] & H^{i+1+p, b}_{}(\C) \ar[r]\ar[d]& \cdots 
	\\
	\cdots \ar[r] &H^{i+p}_{\sing}(\pt) \ar[r]& H^{i+p\sigma}_{\Br}(\Es C_2) \ar[r] & H^{i+(p+1)\sigma}_{\Br}(\Es C_2) \ar[r] & H^{i+1+p}_{\sing}(\pt) \ar[r]& \cdots .
	\end{tikzcd}
	\]
	A straightforward induction shows that for $p\geq 0$, Betti realization 
	$H^{i+p\sigma, b}_{C_2}(\EG C_2)\to H^{i+p\sigma}_{\Br}(\Es C_2)$ 
	is an isomorphism if $i\leq 2b$ and 
	$H^{i+p\sigma, b}_{C_2}(\EG C_2)=0$ if $i>2b$. This  establishes the result in case $p\geq 0$. 
	
	Now we establish the result for $p\leq 0$.
	  Using the comparison of exact sequences \eqref{eqn:les1} and the five lemma, we find that if the map
	   $H^{i+n\sigma, b}_{C_2}(\EG C_2)\to H^{i+n\sigma}_{\Br}(\Es C_2)$ is an isomorphism for all $i\leq 2b$ when $n=p+1$, then this map is also an isomorphism for all $i\leq 2b$ when $n=p$. By downward induction on $p$, starting with $p=0$, we deduce the computation for $a\leq 2b$. 
	
	Now assume that $H^{2b+1+n\sigma, b}(\EG C_2)=0$ for $n=p+1$. If $p\geq -(2b+1)$, if follows from the exact sequence induced by \eqref{eqn:cof1} that this group vanishes for $n=p$ as well. Thus downward induction  implies the result for $p<-(2b+1)$ once we treat the case $p=-(2b+1)$. Consider the comparison of exact sequences
	\[
	\begin{tikzcd}[column sep=small]
	H^{2b-2b\sigma}_{C_2}(\EG C_2) \ar[r,"\phi"]\ar[d, "\cong"'] & H^{0, b}_{}(\C) \ar[r]\ar[d, "\cong"] & H^{2b+1-(2b+1)\sigma, b }_{C_2}(\EG C_2) \ar[r]\ar[d] & H^{2b+1 -2b\sigma, b }_{C_2}(\EG C_2) = 0 \ar[d]   
	\\
{H^{2b-2b\sigma}_{\Br}(\Es C_2)} \ar[r, "\cong"] & {H^{0}_{\sing}(\pt)} \ar[r]& {H^{2b+1-(2b+1)\sigma}_{\Br}(\Es C_2)} \ar[r] & {H^{2b+1 -2b\sigma}_{\Br}(\Es C_2)}. 
	\end{tikzcd}
	\]
	That the bottom left horizontal arrow is an isomorphism can be seen by noting that this map can be identified with the restriction to the fiber homomorphism $H^{2b}_{\sing}({\rm Th}(\gamma), \Z/n)\to H^{2b}_{\sing}(S^{2b},\Z/n)$, where $\gamma$ is the vector bundle on 
$\Bs C_2$ determined by the $b$-dimensional complex sign representation. This map is an isomorphism because  $\gamma$ is orientable.
	It follows that the map labeled $\phi$ is an isomorphism and so $H^{2b+1-(2b+1)\sigma, b}_{C_2}(\EG C_2) =0$.

	If $a\geq 2b+2$, then $H^{a+p\sigma, b}_{C_2}(\C)\cong H^{a+p\sigma, b}_{C_2}(\EG C_2)$, since $H^{a+p\sigma, b}(\wt\EG C_2)\cong 
	\rH^{a, b}(\Sigma \BG C_2)=0$ for $a\geq 2b+2$. The case $a\geq 2b+2$ thus follows from \aref{prop:ptiso}.

For the last statement about Betti realization, we have already checked that it is an isomorphism if $a\leq 2b$. The remaining part of the statement follows from the commutative diagram, where $2b+2\leq a$ 
	\[
	\begin{tikzcd}
		H^{a+p\sigma, b}_{C_2}(\C) \ar[r, "\cong"]\ar[d, "\cong"'] & H^{a+p\sigma, b }_{C_2}(\EG C_2) \ar[d] \\
		{H^{a+p\sigma}_{\Br}(\pt)} \ar[r, "\cd 2"]& 
		{H^{a+p\sigma}_{\Br}(\Es C_2)}.
	\end{tikzcd}
	\]
To see that the bottom arrow is multiplication by $2$, note that for $2\leq a$,
$H^{a+p\sigma}_{\Br}(\pt, \ul{\Z})\cong H_{-a-p\sigma}^{\Br}(\EG C_2,\ul{\Z})$, see e.g.,~\cite{Greenlees:4} for details, and under this identification the lower arrow is induced by the norm map $H\Z_{hC_2}\to H\Z^{hC_2}$.

 \end{proof}

\begin{proposition} \label{neq} If $b+q<0$ then 
	$H^{\star,
		b+q\sigma} _{C _2}(\EG C _2, \ul{\Z/n}) = 0$.
\end{proposition}
\begin{proof} 
	We have that $H^{a+p\sigma,b+q\sigma} _{C _2}(\EG C _2,\ul{\Z/n})\cong 
H^{a+2q+(p-2q)\sigma,b+q}_{C _2}(\EG C _2, \ul{\Z/n})$. 
	Using  the vanishing $H^{a+2q,b+q}_{C_2}(\EG C _2, \ul{\Z/n}) \cong H^{a+2q,b+q}(\BG C _2, \Z/n) =0$ together with the exact sequence induced by \eqref{eqn:cof1}, the result follows by induction. 
\end{proof}

 \begin{notation}\label{namedelts}
 We introduce certain  elements in the
 cohomology of  $\spec(\C)$ and $\EG C_2$. 
  The stated isomorphisms between the cohomology of  $\spec(\C)$ and $\EG C_2$ all follow from the exact sequences associated to \eqref{eqn:cof2} together with the vanishing of the Bredon motivic cohomology of $\EGt C_2$ in the relevant degrees, see \aref{crt}.
 \begin{itemize} 
 	\item $\tau_\sigma\in H^{0,\sigma} _{C _2}(\C,\ul{\Z/n})\cong
 	 H^{0,\sigma} _{C _2}(\EG C _2,\ul{\Z/n})\cong \Z/n$ is a generator.
 	
 	\item  $\xi\in H^{-2+2\sigma,-1+\sigma} _{C _2}(\C,\ul{\Z/n})\cong H^{-2+2\sigma,-1+\sigma}_{C _2}(\EG C _2,\ul{\Z/n})\cong \Z/n$ is a generator.

 \end{itemize}

\end{notation}

 Next we compute the multiplicative structure.  
The $\MMtn$-modules $H^{\star, b+q\sigma}_{C_2}(\EG C_2,\ul{\Z/n})$ together with multiplicative structure are displayed in \aref{fig:H(EC2)} below.

 \begin{lemma}\label{lem:was4.5}
 	Let $b+q\geq 0$. Then 
 	\[
 	\cdot \tau_\sigma: H^{a+p\sigma, b+q\sigma}_{C_2}(\EG C_2,\ul{\Z/n})\xrightarrow{\cong} H^{a+p\sigma, b+(q+1)\sigma}_{C_2}(\EG C_2,\ul{\Z/n})
 	\]
 	is an isomorphism. 
 \end{lemma}
 \begin{proof}
 	By periodicity, it suffices to treat the case $q=0$.

 	If $a\leq 2b$, the claim follows from \aref{Bocal}, since $\Re(\tau_\sigma)=1$. 
 	 	It holds for $a=2b+1$ as these groups are both zero. 
 	 	If $a\geq 2b+2$, then $\rH^{a+p\sigma, b+q\sigma}_{C_2}(\EGt C_2,\ul{\Z/n})=0$ by \aref{crt} which implies that
 	$H^{a+p\sigma, b+q\sigma}_{C_2}(\C,\Z/n)\cong H^{a+p\sigma, b+q\sigma}_{C_2}(\EG C_2,\Z/n)$.  
 	Multiplication by $\tau_\sigma$ fits into 
 	the commutative triangle
 	\[
 	\begin{tikzcd}
 	H^{a+p\sigma, b}_{C_2}(\C,\ul{\Z/n})\ar[d, "\cong"']\ar[r, "\cd \tau_\sigma"] & H^{a+p\sigma, b+\sigma}_{C_2}(\C,\ul{\Z/n})\ar[dl] \\
 	H^{a+p\sigma}_{\Br}(\pt, \ul{\Z/n}).
 	\end{tikzcd}
 	\]
 	It follows that 
 	$\cd \tau_\sigma: H^{a+p\sigma, b}_{C_2}(\EG C_2,\ul{\Z/n})\to H^{a+p\sigma, b+\sigma}_{C_2}(\EG C_2,\ul{\Z/n})$ is injective. 
 But it follows from \aref{Bocal} that either both of these groups are $\Z/n$ or both are $0$. Thus the map is an isomorphism.
 \end{proof}

\begin{theorem}\label{thm:was4.8}
	Let $n\geq 2$. The canonical map is an isomorphism of rings
\[
\MMtn[\xi^{\pm 1}, \tau_\sigma] \xrightarrow{\cong} H^{\star,\star}_{C_2}(\EG C_2, \ul{\Z/n}).
\]

\end{theorem} 
 \begin{proof}
 	Since $H^{\star,0}_{C_2}(\EGt C_2,\ul{\Z/n}) = 0$, we have $H^{\star,0}_{C_2}(\C,\ul{\Z/n})\cong 
 	H^{\star,0}_{C_2}(\EG C_2,\ul{\Z/n})$.
 	Thus, together with periodicity, we have an isomorphism
 	\[
 	H^{\star,0}_{C_2}(\C, \ul{\Z/n})[\xi^{\pm 1}]\xrightarrow{\cong} \bigoplus_{a,p,b}H^{a+p\sigma, b-b\sigma}_{C_2}(\EG C_2, \ul{\Z/n}).
 	\] 	
 	The result now follows from \aref{lem:was4.5}.
 \end{proof}

\begin{remark}
Recall \eqref{eqn:voebg} that $H^{*,*}_{C_2}(\EG C_2, \ul{\Z/2}) \cong H^{*,*}(\BG C_2, \Z/2)\cong \Z/2[\tau][e_1,e_2]/\{e_1^{2}=e_2\tau\}$, where $|e_1| = (1,1)$ and $|e_2|=(2,1)$.
In terms of the generators that appear in \aref{thm:was4.8} we have
$e_1=\frac{au\tau_\sigma}{\xi},e_2=\frac{a^2\tau_\sigma}{\xi},\tau=\frac{u^2\tau_\sigma}{\xi}$.

\end{remark}

The multiplicative structure of $H^{\star,\star}_{C_2}(\EG C_2)$ is displayed the following figure.

\begin{figure}[H]
	\begin{tikzpicture}[scale=1.7]
		\draw[style=mygrid] (-3.0,-2.0) grid (3.0,2.0);
		\foreach \x in {-3,...,3}{
		\draw  (\x ,-2.5)  node {$\x$};
		}
	
		\foreach \y in {-2,...,2}{
		\draw (-3.5,\y ) node {$\y$};
		}

		\draw[->, thick] (-3.8,1.3) -- (-3.8,1.6) node[above] {$q\cd \sigma$};
		\draw[->, thick] (2.3, -2.8) -- (2.6, -2.8) node[right] {$b$};

		\foreach \x in {-2,...,2}{
		\draw[->, very thick, color=ForestGreen] (\x,-\x) -- (\x,2.3);
	}	
		\draw[->,very thick, color=ForestGreen] (3,-2) -- (3,2.3);
		\draw[->,very thick, red] (0,0) -- (-2.3,2.3);
		\draw[->,very thick, red] (0,1) -- (-1.3,2.3);
		\draw[->,very thick, red] (0,2) -- (-0.3,2.3);
		\draw[->,very  thick, red] (3,2) -- (3.3, 1.7);
		\draw[->, very thick, red] (2,2) -- (3.3, 0.7);
		\draw[->, very thick, red] (1,2) -- (3.3, -0.3);
		\draw[->,very  thick, red] (0,2) -- (3.3, -1.3);
		\draw[->, very thick, red] (0,1) -- (3.3, -2.3);
		\draw[->, very thick, red] (0,0) -- (2.3, -2.3);

		\foreach \x in {-2,...,2}{
			\foreach \y in {-\x,...,2} {
			\node at (\x,\y) [fill=black, circle, inner sep = 2pt] {};
			}
		}
	\foreach \y in {-2,...,2} {
			\node at (3,\y) [fill=black, circle, inner sep = 2pt] {};
			}

		\draw (0,0) node[above right] {$1$};
		\draw (-1,1) node[above right] {$\xi$};
		\draw (1,-1) node[above right] {$\xi^{-1}$};
		\draw (0,1) node[above right] {$\tau_\sigma$};

		\node at (-2.5, -1.45) [fill=black, circle, inner sep = 2pt] {};
		\node at (-2.35,-1.45) [right, fill= white] {$=\MMtn$};

	\end{tikzpicture} \hfill 
	\caption{{$H^{\star, b+q\sigma}_{C_2}(\EG C_2,\ul{\Z/n})$}, the  elements have degrees $|\xi|=(-2+2\sigma, -1+\sigma)$ and $|\tau_\sigma|= (0,\sigma)$. Vertical green lines are multiplication by $\tau_\sigma$, upward diagonal red lines indicate multiplication by $\xi$ and downward diagonal red lines indicate multiplication by $\xi^{-1}$.}
	\label{fig:H(EC2)}
\end{figure}
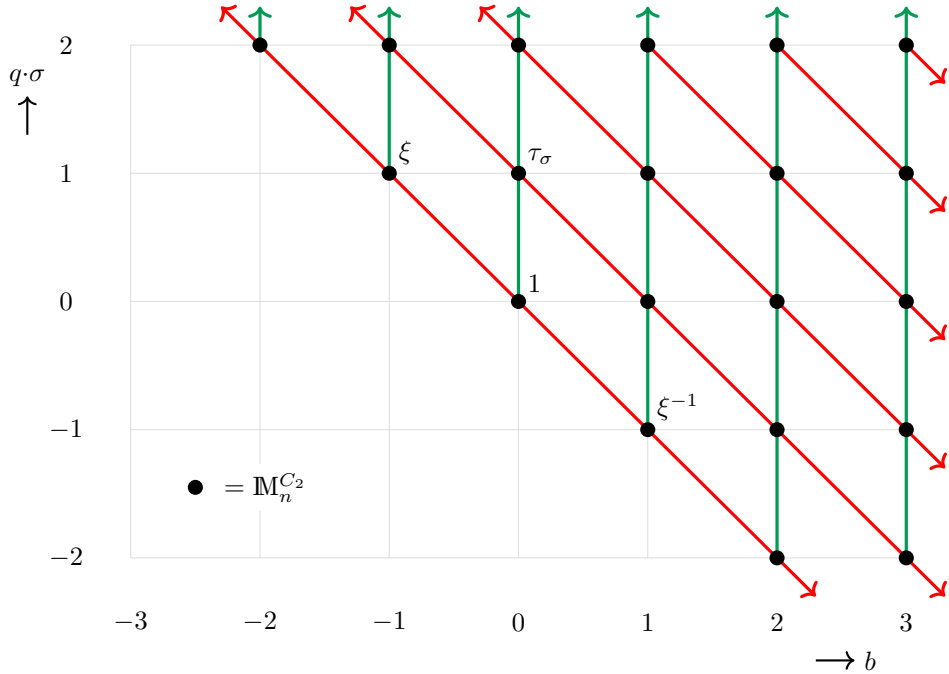

We end this section with a determination of $H^{\star,\star}_{C_2}(\EGt C_2,\ul{\Z/2})$. The $\MMt$-submodules of 
$\rH^{\star}_{\Br}(\wt\Es C_2, \ul{\Z/2})$, defined by 
$\block_i := \Sigma^{2-2\sigma} \Z/2[ a^{\pm 1}, u^{-1}]/(u^{-2i})$ were introduced in \aref{rem:blocks}.
In the following proposition $\mu$ and $\tau_\sigma$ are formal variables which serve the purpose of placing $\block_i$ into the correct weight. The names of these formal variables are chosen to indicate the $H^{\star,\star}_{C_2}(\C,\ul{\Z/2})$-module structure, which will be determined in the next section.

\begin{proposition}\label{prop:wteg}
	There is an isomorphism of 
	$\MMt$-modules
\[
\rH^{\star,\star}(\wt\EG C_2,\ul{\Z/2}) \cong \bigoplus_{i\geq 1, j\in \Z}\hspace{-0.8em}\block_i\hspace{-1pt}\left\{\mu^i\tau_\sigma^j\right\}.
\]
	where $|\tau_\sigma|=(0,\sigma)$ and $|\mu| = (0,1-\sigma)$. 
\end{proposition}
\begin{proof}

It follows from \aref{crt} and \aref{lema:leqb} that  Betti realization 
\[
\Re:\rH^{a+p\sigma, b+q\sigma}_{C_2}(\EGt C_2, \ul{\Z/2})\to 
\rH^{a+p\sigma}_{\Br}(\wt \Es C_2, \ul{\Z/2})
\]
is an isomorphism if $b\geq 0$ and $a\leq 2b+1$  
and $\rH^{a+p\sigma, b+q\sigma}_{C_2}(\EGt C_2, \ul{\Z/2}) = 0$ if $b\leq 0$ or $a>2b+1$. In particular we see that for $b\geq 1$
Betti realization identifies 
$\rH^{\star, b+q\sigma}_{C_2}(\EGt C_2)$ 
with the submodule 
\[
\bigoplus_{i\leq 2b+1 } \rH^{i+*\sigma}_{\Br}(\wt \Es C_2) \subseteq  
\rH^{\star}_{\Br}(\wt \Es C_2). 
\]
This is precisely the submodule 
$B_b=\Sigma^{2-2\sigma}\Z/2[u^{-1}, a^{\pm 1}]/u^{-2b}$, see \aref{rem:blocks}.

\end{proof}

%\begin{remark}\label{rem:wtEGmod}
%The same argument shows that 
%Betti realization identifies 
%$\rH^{\star, b+q\sigma}_{C_2}(\EGt C_2,\ul{\Z/n})$ 
%with 
%$\bigoplus_{i\leq 2b+1 } \rH^{i+*\sigma}_{\Br}(\wt \Es C_2,\ul{\Z/n}) \subseteq  
%\rH^{\star}_{\Br}(\wt \Es C_2,\ul{\Z/n})$. 
%
%\end{remark}
	We equip 
	$\bigoplus_{i\geq 1, j\in \Z}\block_i\hspace{-1pt}\left\{\mu^i\tau_\sigma^j\right\}
	$
	with a $\MMt[\xi,\tau_\sigma, \mu]/(\xi\mu-u^2)$-module structure as follows, 
	\begin{itemize}[label={}]
		\item $\cd\tau_\sigma:\block_i\{\mu^i\tau_\sigma^j\}\to \block_i\{\mu^i\tau_\sigma^{j+1}\}$ is the identity.
		\item $\cd\mu:\block_i\{\mu^i\tau_\sigma^j\}\to \block_{i+1}\{\mu^{i+1}\tau_\sigma^j\}$ is the inclusion \eqref{eqn:incl}.
		\item $\cd \xi: \block_i\{\mu^i\tau_\sigma^j\} \to \block_{i-1}\{\mu^{i-1}\tau_\sigma^j\}$, for $i\geq 2$, is \eqref{eqn:byu}, multiplication by $u^2$. 
%		\[
%		\Sigma^{2-2\sigma}\Z/2[u^{-1}, a^{\pm 1}]/u^{-2i} \xrightarrow{\cd u^2}
%		\Sigma^{2-2\sigma}\Z/2[u^{-1}, a^{\pm 1}]/u^{-2i+2},
%		\]
%		i.e., $\xi\cd\frac{\theta}{u^m}a^n\mu^i\tau_\sigma^j = u^2\frac{\theta}{u^m}a^n\mu^{i-1}\tau_\sigma^j$.
	\end{itemize}

We'll see in the next section that this describes the action of $H^{\star,\star}_{C_2}(\C)$. 
The cohomology of $\EGt C_2$, together with the module structure just described,  is displayed in the following figure.
\begin{figure}[H]
	\begin{tikzpicture}[scale=1.7]
		\draw[style=mygrid] (-2.0,-2.0) grid (4.0,2.0);
		\foreach \x in {-2,...,4}{
		\draw  (\x ,-2.5)  node {$\x$};
		}
	
		\foreach \y in {-2,...,2}{
		\draw (-2.5,\y ) node {$\y$};
		}

		\draw[->, thick] (-2.8,1.3) -- (-2.8,1.6) node[above] {$q\cd \sigma$};
		\draw[->, thick] (3.3, -2.8) -- (3.6, -2.8) node[right] {$b$};

		\foreach \x in {1,...,4}{
		\draw[->, very thick, color=ForestGreen] (\x,-2) -- (\x,2.3);
	}	

		\draw[->,very  thick, blue] (4,2) -- (4.3, 1.7);
		\draw[->, very thick, blue] (3,2) -- (4.3, 0.7);
		\draw[->, very thick, blue] (2,2) -- (4.3, -0.3);
		\draw[->, very thick, blue] (1,2) -- (4.3, -1.3);
		\draw[->,very  thick, blue] (1,1) -- (4.3, -2.3);
		\draw[->, very thick, blue] (1,0) -- (3.3, -2.3);
		\draw[->, very thick, blue] (1,-1) -- (2.3, -2.3);
		\draw[->, very thick, blue] (1,-2) -- (1.3, -2.3);
		\path[red,bend right]  (2,-2) edge node[ midway,above, xshift=3pt]{$u^2$}(1,-1) ;
		\path[red, bend right]  (3,-2) edge node[ midway,above, xshift=3pt]{$u^2$}(2,-1) ;
		\path[red, bend right]  (2,-1) edge node[ midway,above, xshift=3pt]{$u^2$}(1,0) ;
		\path[red, bend right]  (4,-2) edge (3,-1) ;
		\path[red, bend right]  (3,-1) edge (2,0) ;
		\path[red,bend right]  (2,0) edge (1,1) ;
		\path[red,bend right]  (4,-1) edge (3,0) ;
		\path[red, bend right]  (3,0) edge (2,1) ;
		\path[red,bend right]  (2,1) edge (1,2) ;
		\path[red,bend right]  (3,1) edge (2,2) ;
		\path[red, bend right]  (4,0) edge (3,1) ;
		\path[red, bend right]  (4,1) edge (3,2) ;
	\foreach \x in {1,...,4}{
			\foreach \y in {-2,...,2}{
				\node at (\x,\y) [draw, fill = white] {$\x$};
			}
		}

	\node at (1,-2) [above right, inner sep = 7pt] {$\mu\tau_{\sigma}^{-1}$};
\node at (2,-2) [above right, inner sep = 7pt] {$\mu^2$};
\node at (3,-2) [above right, inner sep = 7pt] {$\mu^3\tau_{\sigma}$};
\node at (2,-1) [above right, inner sep = 7pt] {$\mu^2\tau_{\sigma}$};
\node at (4,-2) [above right, inner sep = 7pt] {$\mu^4\tau_{\sigma}^{2}$};
\node at (1,-1) [above right, inner sep = 7pt] {$\mu$};
\node at (1,0) [above right, inner sep = 7pt] {$\mu\tau_{\sigma}$};
\node at (1,1) [above right, inner sep = 7pt] {$\mu\tau_{\sigma}^{2}$};
\node at (1,2) [above right, inner sep = 7pt] {$\mu\tau_{\sigma}^{3}$};
	\node at (-1.85, -1.5) [draw, fill = white] {$i$};
%		\node at (-2.35,-2.8) [right, fill= white] {$=\block_i$};
		\node at (-1.73,-1.5) [right, fill= white] {$=\Sigma^{2-2\sigma}\Z/2[u^{-1}, a^{\pm 1}]/u^{-2i}$};
	\end{tikzpicture} \hfill 
	\caption{{$\rH^{\star, b+q\sigma}_{C_2}(\EGt C_2,\ul{\Z/2})$}. Vertical green lines are multiplication by $\tau_\sigma$, diagonal blue lines are multiplication by $\mu$. The curved red lines indicate the action of $\xi$, which acts as multiplication by $u^2$. }
	\label{fig:H(wtEC2)}
\end{figure}
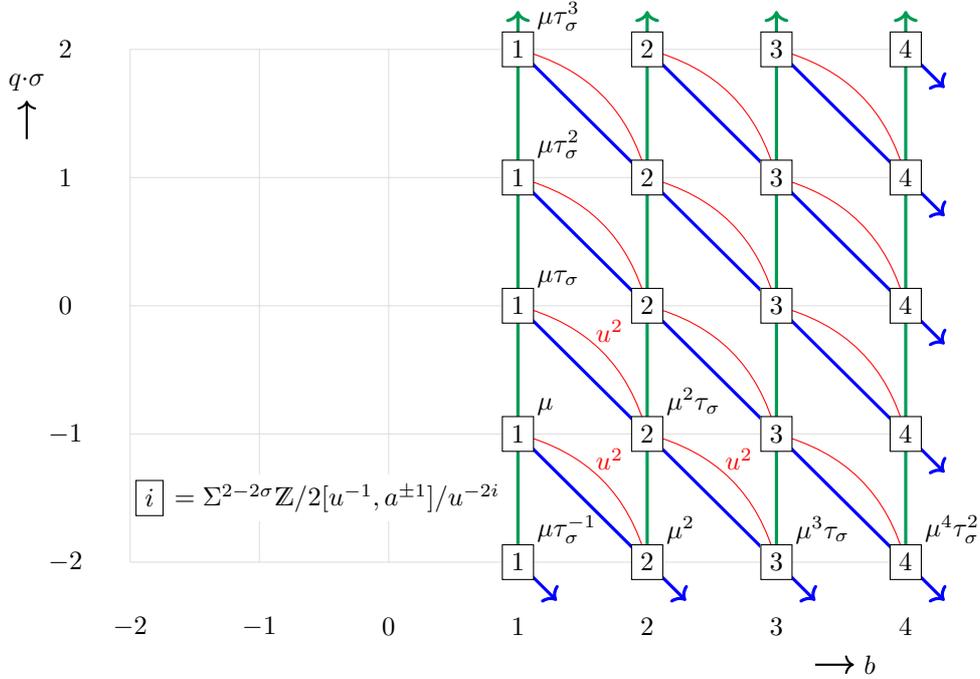

\section{Bredon motivic cohomology of \texorpdfstring{$\C$}{C}}\label{sec:BC}
 
  This section identifies the Bredon motivic cohomology ring of the complex numbers with $\Z/2$ coefficients.

  We begin with some additive structure.
  
  \begin{theorem}\label{thm:additive}
  Let $n\geq 2$ be a natural number.
  	\begin{enumerate}
  		\item If $b\geq 0$ and $b+q\geq 0$ then Betti realization induces an isomorphism 
  		\[
  		\Re : H^{\star,b+q\sigma}_{C_2}(\C,\ul{\Z/n})\xrightarrow{\cong}
  		\MMtn.
  		\] 
  		
  		\item If $b\geq 0$ and $b+q<0$ then 
  		$H^{\star, b+q\sigma}_{C_2}(\C,\ul{\Z/n})\cong 
  		\rH^{\star, b+q\sigma}_{C_2}(\EGt C_2, \ul{\Z/n})$. 
%  		In particular
%  			\[
%  			H^{\star, b+q\sigma}_{C_2}(\C,\ul{\Z/2})\cong \block_b.
%  			\]
%  			Moreover, if $a \leq 2b+1$ and $a< -p$, then  
%  		$\Re$ is an isomorphism, it is injective if $a=-p$, and is zero for all other values of $a,p$.
  		Moreover, $\Re$ is identified with the $\MMtn$-module map
  		\[\bigoplus_{a\leq 2b+1}\rH^{a+*\sigma}_{\Br}(\wt\E C_2, \ul{\Z/n})\to \MMtn.\]
  		
  		\item If $b<0$ and $b+q\geq 0$ then 
  		$H^{\star, b+q\sigma}_{C_2}(\C,\ul{\Z/n})\cong
  		 H^{\star, b+q\sigma}_{C_2}(\EG C_2, \ul{\Z/n})\cong \MMtn$. 		  	
%  		 In particular,
%  		 \[ 
%  		 H^{a+p\sigma, b+q\sigma}_{C_2}(\C,\Z/n)\cong 
%  	 \begin{cases}
%  	 	H^{a+p\sigma}_{\Br}(\pt,\ul{\Z/n}) & a\leq 2b \\
%  	 	0 & a = 2b+1 \\
%  	 	H^{a+2q+(p-2q)\sigma}_{\Br}(\pt, \ul{\Z/n}) & 2b+2\leq a.
%  	 \end{cases}
%  		 \]
  		 
%  		 Moreover, $\Re $ is an isomorphism if $a\leq 2b$ or if $2\leq a$. It is multiplication by $2$ if $2b+2\leq a\leq 1$ where $a$ is even and $a+p=0$, and it is zero for all other values of $a,p$.  
Moreover $\Re$ is identified with the $\MMtn$-algebra map 
\[\MMtn\to H^{\star}_{\Br}(\E C_2, \ul{\Z/n}).\]
  		
  \item If $b<0$ and $b+q<0$, then $H^{\star, b+q\sigma}_{C_2}(\C, \ul{\Z/n})=0$. 
\end{enumerate}
  \end{theorem}
 \begin{proof}
 	
 We make use of the comparison of long exact sequences, obtained from \eqref{eqn:cof1} and \eqref{eqn:cof2} 
 \begin{equation}\label{eqn:6cof1}
 \begin{tikzcd}[column sep=small]
 \cdots \ar[r] & H^{a+p, b +q}_{}(\C) \ar[r]\ar[d] & H^{a+p\sigma, b +q\sigma}_{C_2}(\C) \ar[r]\ar[d] & H^{a+(p+1)\sigma, b +q\sigma}_{C_2}(\C) \ar[r]\ar[d] & H^{a+1+p, b+q}(\C)\ar[r]\ar[d] & \cdots \\
 \cdots \ar[r] & H^{a+p}_{\sing}(\pt) \ar[r]& H^{a+p\sigma}_{\Br}(\pt) \ar[r] & H^{a+(p+1)\sigma}_{\Br}(\pt) \ar[r] & H^{a+1+p}_{\sing}(\pt)\ar[r] & \cdots
 \end{tikzcd} 
 \end{equation}
 and
 \begin{equation}\label{eqn:6cof2}
 \begin{tikzcd}[column sep=small]
 	\cdots \ar[r] & \rH^{a+p\sigma, b +q\sigma}_{C_2}(\EGt C_2) \ar[r]\ar[d] & H^{a+p\sigma, b +q\sigma}_{C_2}(\C) \ar[r]\ar[d] & H^{a+p\sigma, b +q\sigma}_{C_2}(\EG C_{2}) \ar[r]\ar[d] & \cdots 
 	\\
 	\cdots \ar[r] & \rH^{a+p\sigma}_{\Br}(\wt\E C_2) \ar[r]& H^{a+p\sigma}_{\Br}(\pt) \ar[r] & H^{a+p\sigma}_{\Br}(\E C_{2}) \ar[r] & \cdots .
 \end{tikzcd}
 \end{equation}
  
 First we note that (4) follows since $ H^{a+p\sigma, b +q\sigma}_{C_2}(\EG C_{2})=0$ if $b+q <0$ (see \aref{neq}) and
 $\rH^{a+p\sigma, b +q\sigma}_{C_2}(\EGt C_2) = 0$ if $b<0$ (see \aref{crt}).

 To establish (1), we first observe that it 
 suffices to show that $H^{a,b+q\sigma}_{C_2}(\C)\to H^{a}_{\Br}(\pt)$ is an isomorphism for all $a$. Indeed the general case follows from the $p=0$ case by induction (upwards and downwards) on $p$, using \eqref{eqn:6cof1}, since $H^{*,n}(\C)\to H^*_{\sing}(\pt)$ is an isomorphism when $n\geq 0$. 
 Next, we note that using \eqref{eqn:6cof2} together with \aref{Bocal} and
 \aref{crt}, we have that $H^{a, b +q\sigma}_{C_2}(\C)\to H^{a}_{\Br}(\pt)$ is an isomorphism for $a\leq 2b$ and that $H^{a, b +q\sigma}_{C_2}(\C) =0$ for $a>2b$.
  Since $b\geq 0$, for this implies that $\Re$ is also an isomorphism for $a> 2b$.

 For part (2), consider  the commutative diagram
 \[
 \begin{tikzcd}[column sep=small]
 	\rH^{a+p\sigma, b +q\sigma}_{C_2}(\EGt C_2) \ar[r, "\cong"]\ar[d, hookrightarrow] & H^{a+p\sigma, b +q\sigma}_{C_2}(\C) \ar[d] 
 	\\
 	\rH^{a+p\sigma}_{\Br}(\wt\E C_2) \ar[r]& H^{a+p\sigma}_{\Br}(\pt).
 \end{tikzcd}
 \]
 The upper horizontal arrow is an isomorphism  
 since 
  $H^{\star, b+q\sigma}(\EG C_2, \ul{\Z/n})=0$ if $b+q<0$.  
  It follows from \aref{lema:leqb} and \aref{crt} that 
  Betti realization identifies 
  $\rH^{\star, b+q\sigma}_{C_2}(\EGt C_2)$ 
  with 
  $
  \bigoplus_{i\leq 2b+1 } \rH^{i+*\sigma}_{\Br}(\wt \Es C_2) \subseteq  
  \rH^{\star}_{\Br}(\wt \Es C_2)$.
% If $a+p<0$ then $H^{a+p\sigma}_{\Br}(\E C_2)=0$, so the lower horizontal map is an isomorphism and hence the right vertical map is as well. 
% If $a+p=0$, 
% then the lower horizontal map is injective, implying it is an isomorphism (as the domain and codomain have the same finite number of elements). Hence the right vertical map is an isomorphism as well. 
% Finally, by inspection, we see that if $a+p>0$, the right vertical map is always zero (either the domain or codomain is always zero in this case).  
     
  The statements of (3) follow from \aref{crt} and \aref{thm:was4.8}. 
   
 \end{proof}

% \begin{remark}\label{rem:similar}
For simplicity,  we now restrict to the case of $\Z/2$ coefficients, because this is the most interesting case. 
However, it is straightforward to adapt the following discussion to the general case of $\ul{\Z/n}$-coefficients.
%\end{remark}

Define 	
\[
\mu \in H^{0,1-\sigma}_{C_2}(\C,\ul{\Z/2}) \cong \Z/2
\] 
to be the generator. (Here we use that $H^{0,1-\sigma}_{C_2}(\C)\cong H^{0,1-\sigma}_{C _2}(\EG C _2)$.)
Also recall that we defined elements 
 $\xi\in H^{-2+2\sigma, -1+\sigma}_{C_2}(\C,\ul{\Z/2})$ and 
 $\tau_\sigma\in H^{0, \sigma}_{C_2}(\C,\ul{\Z/2})$
each to be the generator of the displayed group (each of which is equal to $\Z/2$). 

\begin{remark}
	In terms of these elements we have
	\[
	\tau = \mu \tau_\sigma\in H^{0,1}_{C_2}(\C,\ul{\Z/2})\cong H^{0,1}(\C,\Z/2).
	\]
	(this can be seen, for example, by noting that $\Re(\mu \tau_\sigma)=1$).
\end{remark}

\begin{proposition}\label{prop:1pce}
There is an $\MMt$-algebra isomorphism
 		\[
 		\phi:\MMt[\xi,\tau_\sigma, \mu]/(\xi\mu-u^2)\xto{\cong} \bigoplus_{b+q\geq 0}H_{C_2}^{\star,b+q\sigma}(\C,\ul{\Z/2})
 		\]
 		defined by $\xi\mapsto \xi$, $\tau_\sigma\mapsto \tau_\sigma$, $\mu\mapsto \mu$.
\end{proposition}
\begin{proof}
First we note that the relation $\xi\mu = u^2$ holds in 
$H^{\star,0}_{C_2}(\C, \ul{\Z/2})$, so that there is a well-defined $\MMt$-algebra map $\phi$.
To see that this relation holds, it suffices to note that it holds in  $H^{\star,0}_{C_2}(\EG C_2)$ (since $\rH^{\star,0}_{C_2}(\EGt C_2)=0$). Now  $\mu$ is the generator of $H^{0,1-\sigma}_{C _2}(\EG C _2)$, but by periodicity, this group is also generated by $\frac{u^2}{\xi}$ (since $u^2$ generates $H^{-2+2\sigma,0}_{C _2}(\EG C _2)\cong H^{-2+2\sigma}_{\Br}(\pt)$).

Examining the weights of elements, we see that
 \[ 
 \left( \MMt[\xi,\tau_\sigma, \mu]/(\xi\mu-u^2)\right)^{(\star, b+q\sigma)} \cong \begin{cases}
 	\MMt\cdot\{\mu^b\tau_\sigma^{q+b} \} & b\geq 0 \vspace{2mm}\\
 	\MMt\cdot \{\xi^b\tau_\sigma^{q-b}\} & b\leq 0.
 \end{cases}
 \]
  
 If $b\geq 0$,  consider the commutative diagram
 \[
 \begin{tikzcd}
 	\left( \MMt[\xi,\tau_\sigma, \mu]/(\xi\mu-u^2)\right)^{(\star, b+q\sigma)}\ar[r, "\phi"]\ar[dr, "\cong"'] & H^{\star,b+q\sigma}_{C_2}(\C, \ul{\Z/2})\ar[d, "\cong"',"\Re"] \\
 	& \MMt.
 \end{tikzcd}
 \]
 The right vertical arrow is an isomorphism by \aref{thm:additive} and thus so is $\phi$. If $b\leq 0$, consider the commutative diagram
 \[
 \begin{tikzcd}
 	\left( \MMt[\xi,\tau_\sigma, \mu]/(\xi\mu-u^2)\right)^{(\star, b+q\sigma)}\ar[r, "\phi"]\ar[dr, "\cong"'] & H^{\star,b+q\sigma}_{C_2}(\C, \ul{\Z/2})\ar[d, "\cong"] \\
 	& H^{\star, b+q\sigma}_{C_2}(\EG C_2,\ul{\Z/2}). 
 \end{tikzcd}
 \]
 The vertical arrow is an isomorphism and therefore, so is $\phi$. 
\end{proof}

Next we verify that  the module structure on $\rH^{\star,\star}_{C_2}(\EGt C_2,\ul{\Z/2})$, displayed in \aref{fig:H(wtEC2)} is the one induced by the isomorphism of the previous proposition. 

\begin{lemma}\label{lem:wtmod}
The $\MMt[\xi,\tau_\sigma, \mu]/(\xi\mu-u^2)$-module structure on $\rH^{\star,\star}_{C_2}(\EGt C_2,\ul{\Z/2})$, induced by the isomorphism of \aref{prop:1pce}, is determined as follows,
	\begin{itemize}[label={}]
		\item $\cd\tau_\sigma:\block_i\{\mu^i\tau_\sigma^j\}\to \block_i\{\mu^i\tau_\sigma^{j+1}\}$ is the identity.
		\item $\cd\mu:\block_i\{\mu^i\tau_\sigma^j\}\to \block_{i+1}\{\mu^{i+1}\tau_\sigma^j\}$ is the inclusion \eqref{eqn:incl}.
		\item $\cd \xi: \block_i\{\mu^i\tau_\sigma^j\} \to \block_{i-1}\{\mu^{i-1}\tau_\sigma^j\}$, for $i\geq 2$, is \eqref{eqn:byu}, multiplication by $u^2$.
	\end{itemize} 
\end{lemma}
\begin{proof}

Given an element $x\in H^{a+p\sigma, m+n\sigma}_{C_2}(\C,\ul{\Z/2})$, we have the commutative square 
\[
\begin{tikzcd}
\rH^{\star,b+q\sigma}_{C_2}(\EGt C_2, \ul{\Z/2}) \ar[r, "\cd x"]\ar[d, hookrightarrow] & \rH^{\star,b+m+(q+n)\sigma}_{C_2}(\EGt C_2, \ul{\Z/2}) \ar[d,  hookrightarrow] \\
\rH^\star_\Br(\wt\Es C_2,\ul{\Z/2}) \ar[r, "\cd \Re(x)"] & \rH^\star_\Br(\wt\Es C_2,\ul{\Z/2}).
\end{tikzcd}
\]
The identification of the module structure follows using that $\Re(\tau_\sigma) =\Re(\mu) = 1$ and $\Re(\xi) = u^2$.
\end{proof}

\begin{theorem}\label{thm:ptring}
 		There is an 
 		$\MMt$-algebra isomorphism 
\[
\MMt[\xi,\tau_\sigma, \mu]/(\xi\mu-u^2)\oplus\Big(\bigoplus_{i,j\geq 1}\hspace{-0.1em}\block_i\hspace{-1pt}\left\{\frac{\mu^i}{\tau_\sigma^j}\right\}\Big)
\xto{\cong} 
H_{C_2}^{\star,\star}(\C,\ul{\Z/2}).
\]
defined by $\xi\mapsto \xi$, $\tau_\sigma\mapsto \tau_\sigma$, $\mu\mapsto \mu$.
 		The multiplicative structure on $H_{C_2}^{\star,\star}(\C, \ul{\Z/2})$ is determined as follows.
 		\begin{enumerate}
 			\item The left hand summand is the displayed quotient of a polynomial ring.
 			\item Multiplications between elements in the left hand  and the right hand summands, are determined by 
 			\begin{itemize}[label={}]
 				\item $\cd\tau_\sigma:\block_i\{\mu^i\tau_\sigma^j\}\to \block_i\{\mu^i\tau_\sigma^{j+1}\}$, for $j\geq 2$, is the identity.
 				\item $\cd\tau_{\sigma}: \block_i\{\mu^i\tau_\sigma\}\to\MMt\{\mu^i\}$ is the map \eqref{eqn:BiM}.
 				\item $\cd\mu:\block_i\{\mu^i\tau_\sigma^j\}\to \block_{i+1}\{\mu^{i+1}\tau_\sigma^j\}$ is the inclusion \eqref{eqn:incl}.
 				\item $\cd \xi: \block_i\{\mu^i\tau_\sigma^j\} \to \block_{i-1}\{\mu^{i-1}\tau_\sigma^j\}$ is \eqref{eqn:byu}, multiplication by $u^2$.
% 				\[
% 				\Sigma^{2-2\sigma}\Z/2[u^{-1}, a^{\pm 1}]/u^{-2i} \xrightarrow{\cd u^2}
% 				\Sigma^{2-2\sigma}\Z/2[u^{-1}, a^{\pm 1}]/u^{-2i+2}.
% 				\]
 			\end{itemize}
 			\item Products in the right hand summand are  trivial.
		\end{enumerate}
 \end{theorem}
 \begin{proof}
 In \aref{prop:1pce}, we have already identified 
 $\oplus_{b+q\geq 0}H^{\star, b+q\sigma}_{C_2}(\C)$.
To identify the remaining piece, we use that the map
  $\oplus_{b+q<0}\rH^{\star,b+q\sigma}_{C_2}(\wt\EG C_{2}, \ul{\Z/2})\xto{\cong} \oplus_{b+q<0}H^{\star, b+q\sigma}_{C_2}(\C, \ul{\Z/2})$ is an isomorphism by \aref{neq}. 
Using \aref{prop:wteg}, we conclude the $\MMt$-module structure on 
$H^{\star, b+q\sigma}_{C_2}(\C, \ul{\Z/2})$ is as in the theorem. 
Products in $\rH^{\star,\star}_{C_2}(\wt\EG C_2, \ul{\Z/2})$ are trivial, since they are determined by products in $\rH^{*,*}_{C_2}(\wt\EG C_2, \ul{\Z/2})\cong \rH^{*,*}(\Sigma \BG C_2, \Z/2)$, which are trivial. 
The multiplicative structure involving both the left and right hand summands follows  from 
 \aref{lem:wtmod} with the exception of the $\tau_\sigma$-multiplications starting in weights $i-(i+1)\sigma$, which follow from the commutative diagram
\[
\begin{tikzcd}
\rH^{\star, i-(i+1)\sigma}_{C_2}(\EGt C_2)\ar[d, hookrightarrow]\ar[r, "\cong"] & H^{\star,i-(i+1)\sigma}_{C_2}(\C) \ar[d]\ar[r, "\cd \tau_\sigma"] &  H^{\star,i-i\sigma}_{C_2}(\C) \ar[d, "\cong"] \\
\rH^\star_\Br(\wt\Es C_2) \ar[r] & \MMt \ar[r, "\cd 1"]& \MMt.
\end{tikzcd}
\]

 \end{proof}

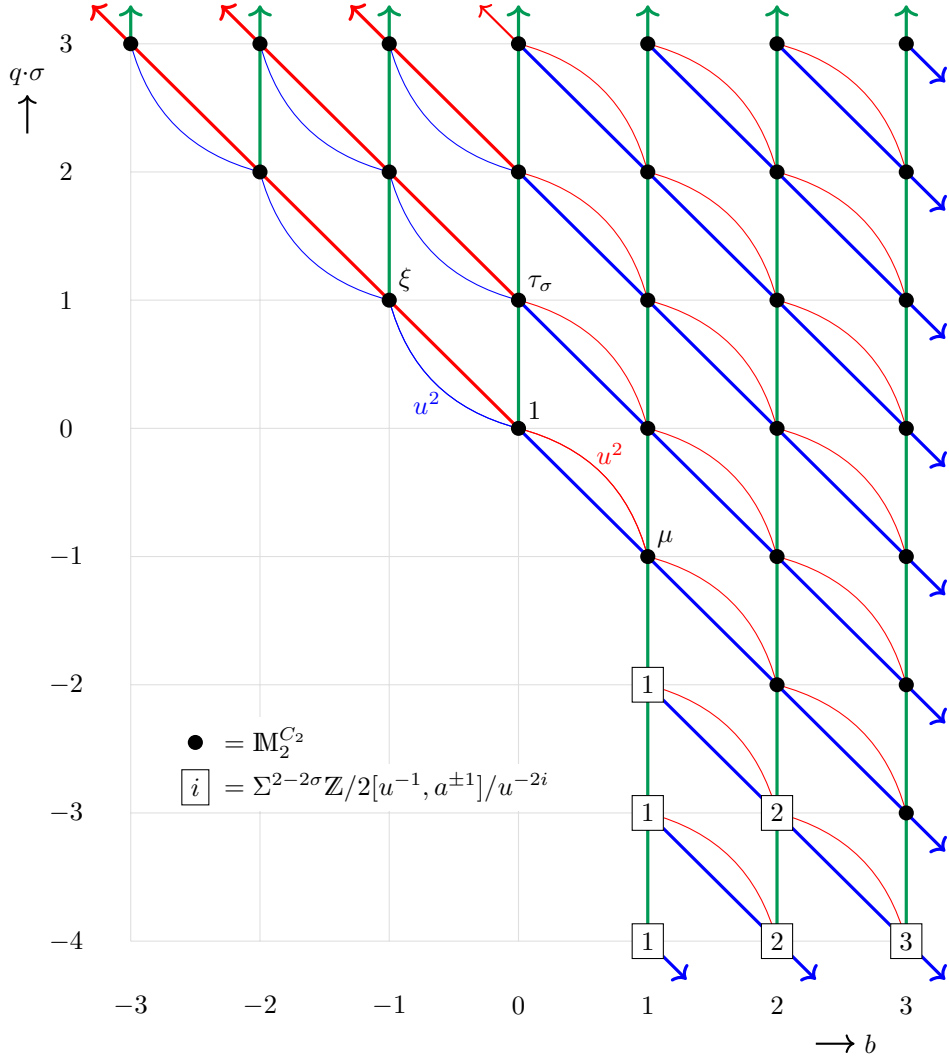
\begin{figure}[H]
	\begin{tikzpicture}[scale=1.7]
		%		\path[style=myfill] (0,0) -- (-4.0,4.0) -- (0,4.0) -- (0,0) ;
		%		\path[style=myfill] (2,-2) -- (2,-4.0) -- (4.0,-4.0) -- (2,-2) ;
		\draw[style=mygrid] (-3.0,-4.0) grid (3.0,3.0);

		\foreach \x in {-3,...,3}{
			\draw  (\x ,-4.5)  node {$\x$};
		}
		
		\foreach \y in {-4,...,3}{
			\draw (-3.5,\y ) node {$\y$};
		}
		
		%		\node at (-3.5, 3.4) {$q\cd\sigma$};
		%		\node at (3.4, -4.5) {$b$};
		\draw[->, thick] (-3.8,2.3) -- (-3.8,2.6) node[above] {$q\cd \sigma$};
		\draw[->, thick] (2.3, -4.8) -- (2.6, -4.8) node[right] {$b$};

		\foreach \x in {-3,...,3}{
			\draw[->, very thick, color=ForestGreen] (\x,-\x) -- (\x,3.3);
		}	
		\draw[->, very thick, color=ForestGreen] (1,-4) --(1,-2);
		\draw[->, very thick, color=ForestGreen] (2,-4) -- (2,-3);
		\draw[ very thick,  color=ForestGreen] (1,-2) -- (1,-1);
		\draw[very  thick, color=ForestGreen] (2,-3) -- (2,-2);
		\draw[ very thick,  color=ForestGreen] (3,-4) -- (3,-3);
		\draw[->,very thick, red] (0,0) -- (-3.3,3.3);
		\draw[->,very thick, red] (0,1) -- (-2.3,3.3);
		\draw[->,very thick, red] (0,2) -- (-1.3,3.3);
		\draw[->,thick, red] (0,3) -- (-0.3,3.3);
		\draw[->,very  thick, blue] (3,3) -- (3.3, 2.7);
		\draw[->, very thick, blue] (2,3) -- (3.3, 1.7);
		\draw[->, very thick, blue] (1,3) -- (3.3, 0.7);
		\draw[->, very thick, blue] (0,3) -- (3.3, -0.3);
		\draw[->,very  thick, blue] (0,2) -- (3.3, -1.3);
		\draw[->, very thick, blue] (0,1) -- (3.3, -2.3);
		\draw[->, very thick, blue] (0,0) -- (3.3, -3.3);
		\draw[->, very thick, blue] (1,-2) -- (3.3, -4.3);
		\draw[->,very  thick, blue] (1,-3) -- (2.3, -4.3);
		\draw[->,very  thick, blue] (1,-4) -- (1.3, -4.3);
		
		\foreach \x in {1,2}{
		\foreach \y in {-3,...,3}{
			\path[red, bend left]  (\x,\y) edge (\x+1,\y-1) ;
		}	
	}
		\foreach \y in {0,...,3}{
			\path[red, bend left]  (0,\y) edge (1,\y-1) ;
		}	
		\foreach \x in {0,...,-2}{
		\foreach \y in {-\x,...,2}{
			\path[blue, bend left]  (\x,\y) edge (\x-1,\y+1) ;
			}	
		}
		\path[red,bend left]  (0,0) edge node[ midway,above, xshift=3pt]{$u^2$}(1,-1);
		\path[blue,bend left]  (0,0) edge node[ midway,below, xshift=-3pt]{$u^2$}(-1,1);

		\foreach \x in {-3,...,3}{
			\foreach \y in {-\x,...,3} {
				%				\fill (\x,\y) circle(2pt);
				\node at (\x,\y) [fill=black, circle, inner sep = 2pt] {};
			}
		}
		\foreach \x in {1,...,3}{
			\pgfmathsetmacro{\newx}{-1-\x};
			\foreach \y in {-4,...,\newx}{
				\node at (\x,\y) [draw, fill = white] {$\x$};
			}
		}

		\draw (0,0) node[above right] {$1$};
		\draw (-1,1) node[above right] {$\xi$};
		\draw (1,-1) node[above right] {$\mu$};
		\draw (0,1) node[above right] {$\tau_\sigma$};

		\node at (-2.5, -2.45) [fill=black, circle, inner sep = 2pt] {};
		\node at (-2.35,-2.4 5) [right, fill= white] {$=\MMt$};
		\node at (-2.5, -2.8) [draw, fill = white] {$i$};
		
		\node at (-2.35,-2.8) [right, fill= white] {$= \Sigma^{2-2\sigma}\Z/2[u^{-1}, a^{\pm 1}]/u^{-2i}$};
	\end{tikzpicture} \hfill 
	\caption{{$H^{\star, b+q\sigma}_{C_2}(\C,\ul{\Z/2})$.} Vertical green lines indicate $\tau_\sigma$-multiplication, upward diagonal lines (red) are $\xi$-multiplication and downward diagonal lines (blue) indicate $\mu$-multiplication, and curved lines record the relation $\xi\mu = u^2$. }
	\label{fig:H(C)}
\end{figure}

\section{Bredon motivic cohomology of algebraically closed fields} 
   In this section we consider an algebraically closed field $k$ and a natural number $n>1$ coprime to $char(k)$. 
   Let $V$ be a $C_2$-equivariant smooth scheme over $k$.
   First we note a rigidity theorem for rational points:
   
   \begin{theorem} \label{sec} For a connected smooth scheme $X$ over $k$ and $k$-rational points $x_{0}$, $x_{1}$ of $X$, 
   $$(x_ 0) _*=(x_ 1) _*: H^{\star,\star} _{C _2}(V\times X,\ul{\Z/n})\rightarrow H^{\star,\star} _{C _2}(V,\ul{\Z/n}).$$ 
   \end{theorem}  
 
 According to \cite{YO}, Theorem \ref{sec} follows if the functor $F(-)=H^{\star,\star} _{C _2}(V\times -,\Z/n)$ is a homotopy invariant presheaf on $Sm/k$ with weak transfers in the sense of \cite{YP}. The four conditions that need to be fulfilled according to \cite{YO}  are:
 
 1) Additivity: For $X=X_0\sqcup X _1$ with corresponding embeddings $i _m:X _m\hookrightarrow X$ for $m=0,1$ and $f: X\rightarrow Y$ a map in $Sm/k$, 
     we have $f _*=(fi _0) _*i _0^*+(fi _1) _*i _1^*$
 
 2) Base change: For every finite flat map $f$, closed embedding $g$, and cartesian diagram: 
 
  \[
\xymatrixrowsep{1cm}
\xymatrixcolsep {1cm}
\xymatrix@-0.9pc{ 
X' \ar[r]^{g _1}\ar[d]^{f _1}& 
Y'\ar[d]^{f} 
\\
X\ar[r]^{g} &
Y
}
\]

 we have $g^*f _*=f_{1*}{g _1}^*$
 
3) Normalization: If $f$ is the identity map on $k$ then $f _*=id_{H^{\star,\star}_{C_2}(V,\ul{\Z/n})}$

4) Homotopy invariance: The rational points 0 and 1 of the affine line $\mathbb{A}^1 _k$ with trivial $C _2$-action yield equal pullback maps
$$0_*=1_*:H^{\star,\star} _{C _2}(V\times _k\mathbb{A}^1 _k)\rightarrow H^{\star,\star} _{C _2}(V).$$
 
 The functor $F$ fulfills all four conditions above as it is a homotopy invariant presheaf with equivariant transfers (\cite{HOV1}). 
 Moreover, because it is a homotopy invariant presheaf with equivariant transfers, according to \cite{Sus}, \cite{YO}, \cite{YP}, 
 we have the following theorem:
 
 \begin{theorem} \label{ext} Suppose $k\subset K$ is an extension of algebraically closed fields and $X$ is a smooth $C_2$-equivariant scheme. 
 If $n$ is coprime to $char(k)$, then $\pi: \spec(K)\rightarrow \spec(k)$ induces an isomorphism: 
 $$\pi^*: H^{\star,\star} _{C _2}(X,\ul{\Z/n})\cong H^{\star,\star} _{C _2}(X _K,\ul{\Z/n})$$
 \end{theorem}
 \begin{proof} We can write $\spec(K)=\lim _U (U)$, where $U$ is an affine smooth variety over $\spec(k)$. There is an induced map $$\pi^*:H^{\star,\star} _{C _2}(X,\ul{\Z/n})\rightarrow H^{\star,\star} _{C _2}(X\times K,\ul{\Z/n})=\colim_U H^{\star,\star} _{C _2}(X\times U,\ul{\Z/n})$$ so if $\pi^*(x)=0$ then there exists a map $\phi: U\rightarrow \spec (k)$ such that $\phi^*(x)=0$. Because $U$ has a $k$-rational point, $\phi$ yields a splitting and $\phi^*$ is injective. This implies $x=0$ so $\pi^*$ is injective.
 
 Next we show that $\pi^*$ is surjective. For every $\beta\in H^{a+p\sigma,b+q\sigma} _{C _2}(X\times K)$ there exists a map $\phi: \spec (K)\rightarrow U$ such that 
 $\phi^*(\beta ')=\beta$ with $\beta '\in H^{\star,\star} _{C _2}(X\times U)$. 
 If $\xi: \spec(k)\rightarrow U$ a rational point, 
 the maps $\xi\circ \pi$, $\phi: \spec (K)\rightarrow U$ induce $K$-rational points $\phi',\xi': \spec(K)\rightarrow U _K$. According to Theorem \ref{sec} we have that
 $$\phi^{'*}=\xi^{'*}: H^{\star,\star} _{C _2}(X\times U\times K)\rightarrow H^{\star,\star} _{C _2}(X\times K).$$
 For the base change $\underline{\pi}:U _K\rightarrow U$, we have 
$\beta- \pi^*\circ \xi^*(\beta')=\phi^*(\beta')-\pi^*\circ \xi^*(\beta ')=(\phi^{'*}-\xi^{*'})(\underline{\pi}^*(\beta'))=0$, and thus $\beta\in Im(\pi^*)$. 
 \end{proof}
 
 The next corollary computes Bredon  motivic cohomology for different algebraically closed fields of characteristic zero.
 \begin{corollary} Let  $K$ an algebraically closed field of characteristic zero and $n>1$. Then $$H^{a+p\sigma,b+q\sigma} _{C _2}(K,\ul{\Z/n})\cong H^{a+p\sigma,b+q\sigma} _{C _2}(\C,\ul{\Z/n})$$
 and $$H^{a+p\sigma,b+q\sigma} _{C _2}(\EG C _{2K},\ul{\Z/n})\cong H^{a+p\sigma,b+q\sigma} _{C _2}(\EG C _2,\ul{\Z/n}),$$
   for any choice of integers $a,b,p,q$. 
 \end{corollary}
  
% \bibliographystyle{plain}
% \bibliography{road}

\vspace{20pt}
\scriptsize
	\noindent
	Jeremiah Heller\\
	University of Illinois at Urbana-Champaign\\
	Department of Mathematics\\
	1409 W. Green Street, Urbana, IL 61801\\
	United States\\
	\texttt{jbheller@illinois.edu}

\vspace{10pt}
	\noindent
	Mircea Voineagu\\
	UNSW Sydney\\
	Department of Mathematics and Statistics\\
	NSW 2052 Australia\\
	\texttt{m.voineagu@unsw.edu.au}

\vspace{10pt}
	\noindent
	Paul Arne {\O}stv{\ae}r\\
	Department of Mathematics F. Enriques \\
	University of Milan \\
	Italy \\
	\texttt{paul.oestvaer@unimi.it}  \\
	%\& \\
	Department of Mathematics\\
	University of Oslo\\
	Norway\\
	\texttt{paularne@math.uio.no}

\end{document}